\documentclass[12pt]{article}

\usepackage[utf8]{inputenc} 
\usepackage[T1]{fontenc}  
\usepackage{amsmath}
\usepackage{amssymb}
\usepackage{amsfonts}
\usepackage{url, enumerate}
\usepackage{hyperref}

\def\udwo{\mbox{UDWO}}

\newcommand{\eop}{\bigstar}  

\newcommand{\cf}{{\rm cf}}

\newenvironment{proof}{\noindent{\bf Proof.}}{\par\bigskip}

\newenvironment{proof-}{\noindent{\bf Proof}}{\par\bigskip}

\newtheorem{THEOREM}{Theorem}[section]

\newtheorem{Conclusion}[THEOREM]{Conclusion}

\newtheorem{Hypothesis}[THEOREM]{Hypothesis}

\newtheorem{LEMMA}[THEOREM]{Lemma}

\newtheorem{Main Theorem}[THEOREM]{Main Theorem}
\newenvironment{main Theorem}{\begin{Main Theorem}} 
{\end{Main Theorem}}

\newtheorem{Theorem}[THEOREM]{Theorem}
\newenvironment{theorem}{\begin{Theorem}}{\end{Theorem}}

\newtheorem{Definition}[THEOREM]{Definition}
\newenvironment{definition}{\begin{Definition}}{\end{Definition}}

\newtheorem{Conventions}[THEOREM]{Conventions}

\newtheorem{Main Definition}[THEOREM]{Main Definition}
\newenvironment{main definition}{\begin{Main Definition}}
{\end{Main Definition}}

\newtheorem{Lemma}[THEOREM]{Lemma}
\newenvironment{lemma}{\begin{Lemma}}{\end{Lemma}}

\newtheorem{Notation}[THEOREM]{Notation}

\newtheorem{Convention}[THEOREM]{Convention}
\newenvironment{convention}{\begin{Convention}}{\end{Convention}}

\newtheorem{Note}[THEOREM]{Note}

\newtheorem{Observation}[THEOREM]{Observation}
\newenvironment{observation}{\begin{Observation}}
{\end{Observation}}

\newtheorem{Remark}[THEOREM]{Remark}

\newtheorem{Question}[THEOREM]{Question}
\newenvironment{question}{\begin{Question}}{\end{Question}}

\newtheorem{Main Fact}[THEOREM]{Main Fact}
\newenvironment{main Fact}{\begin{Main Fact}}{\end{Main Fact}}

\newtheorem{Fact}[THEOREM]{Fact}

\newtheorem{Subfact}[THEOREM]{Subfact}

\newtheorem{Claim}[THEOREM]{Claim}

\newtheorem{Main Claim}[THEOREM]{Main Claim}
\newenvironment{main claim}{\begin{Main Claim}}{\end{Main Claim}}

\newtheorem{Crucial Claim}[THEOREM]{Crucial Claim}
\newenvironment{crucial claim}{\begin{Crucial Claim}}{\end{Crucial Claim}}

\newtheorem{Subclaim}[THEOREM]{Subclaim}

\newtheorem{Sublemma}[THEOREM]{Sublemma}

\newtheorem{Corollary}[THEOREM]{Corollary}
\newenvironment{corollary}{\begin{Corollary}}{\end{Corollary}}

\newtheorem{Example}[THEOREM]{Example}
\newenvironment{example}{\begin{Example}}{\end{Example}}

\newtheorem{Problem}[THEOREM]{Problem}

\newtheorem{Proposition}[THEOREM]{Proposition}

\newtheorem{Conjecture}[THEOREM]{Conjecture}

\newtheorem{Discussion}[THEOREM]{Discussion}

\newenvironment{Proof of the Subfact}
{\noindent{\bf Proof of the Subfact.}}{\par\bigskip}

\newenvironment{Proof of the Theorem}
{\noindent{\bf Proof of the Theorem.}}{\par\bigskip}

\newenvironment{Proof of the Proposition}
{\noindent{\bf Proof of the Proposition.}}{\par\bigskip}

\newenvironment{Proof of the Conclusion}
{\noindent{\bf Proof of the Conclusion.}}{\par\bigskip}

\newenvironment{Proof of the Observation}
{\noindent{\bf Proof of the Observation.}}{\par\bigskip}

\newenvironment{Proof of the Fact}
{\noindent{\bf Proof of the Fact.}}{\par\bigskip}

\newenvironment{Proof of the Lemma}
{\noindent{\bf Proof of the Lemma.}}{\par\bigskip}

\newenvironment{Proof of the Claim}
{\noindent{\bf Proof of the Claim.}}{\par\bigskip}

\newenvironment{Proof of the Corollary}
{\noindent{\bf Proof of the Corollary.}}{\par\bigskip}

\newenvironment{Proof of the Subclaim}
{\noindent{\bf Proof of the Subclaim.}}{\par\medskip}

\newenvironment{Proof of the Main Claim}
{\noindent{\bf Proof of the Main Claim.}}{\par\bigskip}

\newenvironment{Proof of the Crucial Claim}
{\noindent{\bf Proof of the Crucial Claim.}}{\par\bigskip}

\newcommand{\into}{\rightarrow}

\newcommand{\satisfies}{\vDash}

\newcommand{\CC}{{\cal C}}

\newcommand{\EE}{{\cal E}}

\newcommand{\KK}{{\cal K}}
\newcommand{\LL}{{\cal L}}

\def\xb{\bar{x}}

\def\empty{\emptyset}

\newcount\skewfactor
\def\mathunderaccent#1#2 {\let\theaccent#1\skewfactor#2
\mathpalette\putaccentunder}
\def\putaccentunder#1#2{\oalign{$#1#2$\crcr\hidewidth
\vbox to.2ex{\hbox{$#1\skew\skewfactor\theaccent{}$}\vss}\hidewidth}}

\newcommand{\nmodels}{\nvDash}

\newcommand{\dom}{\mbox{\rm dom}}

\newcommand{\ran}{\mbox{\rm ran}}

\newcommand{\ma}{\mathfrak A}
\newcommand{\mb}{\mathfrak B}
\newcommand{\dg}{\rm BG}

\author{Mirna D\v zamonja (\texttt{mdzamonja@irif.fr})\\
    IRIF (CNRS \& Université de Paris)\\
 8 Place Aurélie Némours, 75205 Paris Cedex 13, France\\
 and\\
Jouko V\"a\"an\"anen (\texttt{jouko.vaananen@helsinki.fi})\\
Department of Mathematics University of Helsinki, Helsinki, Finland,\\
and Institute for Logic, Language and Computation\\
University of Amsterdam, Netherlands\\
}

\title{Chain Logic and Shelah's Infinitary Logic}
\begin{document}
\maketitle

\begin{abstract} For a cardinal of the form $\kappa=\beth_\kappa$, Shelah's 
logic $L^1_\kappa$ has a characterisation as the maximal logic above $\bigcup_{\lambda<\kappa} L_{\lambda, \omega}$ 
satisfying  a strengthening of the undefinability of well-order. Karp's chain logic \cite{Karpintroduceschain} 
$L^c_{\kappa, \kappa}$ is  known to satisfy the undefinability of well-order and interpolation. We prove that 
if $\kappa$ is singular of countable cofinality, Karp's chain logic \cite{Karpintroduceschain} is above $L^1_\kappa$. 
Moreover, we show that if $\kappa$ is a strong limit of singular cardinals of countable cofinality, the chain logic $L^c_{<\kappa, <\kappa}=\bigcup_{\lambda<\kappa} L^c_{\lambda, \lambda}$
is a maximal logic with chain models to satisfy a version of the undefinability of well-order.

We then show that the chain logic gives a partial solution to Problem 1.4. from Shelah's \cite{Sh797}, which asked 
whether for $\kappa$ singular of countable cofinality there was a logic strictly between 
$ L_{\kappa^+, \omega}$ and $L_{\kappa^+, \kappa^+}$ having interpolation. We show that modulo accepting as the upper bound a model class of $L_{\kappa, \kappa}$,
Karp's  chain logic satisfies the required properties. In addition, we show that 
this chain logic is not $\kappa$-compact, a question that we have asked on various occasions. 
We contribue to further development of chain logic by proving the Union Lemma and identifying the chain-independent fragment of the logic, showing that it still has considerable expressive power.

In conclusion, we have shown that 
the simply defined chain logic emulates the logic $L^1_\kappa$ in satisfying interpolation, undefinability of well-order and maximality with respect to it, and the Union Lemma. In addition it has a 
completeness theorem.\footnote{We thank Stamatis Dimopoulos and Will Boney for discussion involving the comparison of the chain logic and the logic $L_{\kappa,\omega}$. Mirna D\v zamonja gratefully acknowledges the help of Leverhulme Trust through research Fellowship 2014-2015, the University of Helsinki for their hospitality in October 2014 and June 2016 and the School of Mathematics, University of East Anglia where she was a Professor of Mathematics when this paper was written. 
She received funding from the European Union Horizon 2020 research and innovation programme under the Marie Sklodowska-Curie grant agreement No 1010232 and from the GA{\v C}R project EXPRO 20-31529X and RVO: 67985840 of the Czech Academy of Sciences.
 Jouko V\"a\"an\"anen  was supported by the Faculty of Science of the University of Helsinki, the Academy of Finland grant 322795, and the European Research Council Advanced Grant No 101020762.


}
\end{abstract}

\section{Introduction}\label{intro} In 2012
Saharon Shelah \cite{Sh797} introduced the logic $L^1_\kappa$, defined for $\kappa$ such that $\kappa=\beth_\kappa$. The motivation was to find a strong logic which has many nice properties of first order logic, such as Interpolation, and which moreover, like first order logic, has a  
 model-theoretic characterisation in the style of Lindstr{\"o}m. 
 Lindstr{\"o}m characterised \cite{Lindstrom}  first order logic as the maximal abstract logic (a notion defined precisely below) satisfying the Downward L{\"o}wenheim-Skolem theorem to $\aleph_0$ and Compactness. In spite of the abundance of  abstract logics discovered in the 1960s and 1970s, rare are those that generalise first order logic and have a characterisation of this type. Search for logics with
Lindstr{\"o}m-style characterisation is an active area of research, where positive results were obtained in contexts either far from or weaker than first order logic, including various versions of modal logic 
\cite{deRijke},\cite{vanBenthem},  \cite{Kurz}, \cite{Enqvist2}, \cite{Enqvist}, \cite{Pourmahdian}, or for fragments of first order logic \cite{vanBenthemCate}.  Further examples and references are in \cite{JoukoLindstrom}.

Shelah proved (Theorem 3.4 in  \cite{Sh797}) that $L^1_\kappa$ has such a characterisation, namely it is the maximal logic which is above $\bigcup_{\lambda<\kappa}L_{\lambda, \omega}$ and which satisfies a strong form of the Undefinability of Well Order, called SUDWO. This characterisation makes $L^1_\kappa$ a very special infinitary logic, having a rare property that was not previously known for any classical extension of first order logic.

A feature of Shelah's logic is that it is defined via its elementary equivalence relation. In fact, no syntax is known for this logic in the sense of the syntax of logics such as $L_{\kappa,\lambda}$. For a long time this has been a block for further study of $L^1_\kappa$, as it rendered the logic very complicated. In this work we combat that by giving another view of $L^1_\kappa$, which makes it rather clear what the advantages and disadvantages of that logic are.
At the moment of writing this paper, there are several preprints or work in preparation which study $L^1_\kappa$. In the preprint \cite{Sh1101}, Shelah proves that for $\kappa$ a strongly compact cardinal,
two given models $\mathfrak A$ and $\mathfrak B$ are $L^1_\kappa$-equivalent iff for some $\omega$-sequence
$\bar{\mathcal U}$ of $(<\kappa)$-complete ultrafilters, the iterated ultrapowers of $\mathfrak A$ and $\mathfrak B$ by $\bar{\mathcal U}$ are isomorphic.
In joint yet
unpublished work with A. Villaveces \cite{JoukowithVillaveces} the second author has introduced a fragment of $L^1_\kappa$, which does have a simple syntax and which comes within the so called $\Delta$-extension of $L^1_\kappa$. In joint but as yet unpublished work with B. Veli{\v c}kovi{\'c} \cite{JoukowithVelickovic} the second  author 
has introduced a strengthening of $L^1_\kappa$ which has a syntax in the usual sense and which is equivalent to $L^1_\kappa$ in the case that $\kappa=\beth_\kappa$.

Coming back to the motivation for $L^1_\kappa$, it had became apparent early on that infinitary logics beyond $L_{\omega_1,\omega}$ do not permit a Completeness theorem or an Interpolation theorem in the same sense as first order logic. Already in 1964, Carol Karp introduced the concept of a chain model and showed, in co-operation with her students Ellen Cunningham and Judy Green, see \cite{Cunnigham} and \cite{Green}, that for $\kappa$ strong limit of cofinality $\omega$, the infinitary logic $L_{\kappa\kappa}$ behaves quite nicely in chain models. In particular, they proved a Completeness theorem and the Interpolation theorem for chain logic. 
The main point about chain logic which allows it to have many nice properties, is that the notion of a model is changed to the notion of a chain model (see Definition \ref{chainour}). The price 
we pay is that the notion of isomorphism is also changed, being replaced by the notion of chain-isomorphism (see Definition \ref{chainour}). However, both the notion of  chain model and the chain isomorphism are rather simple and 
are also well studied in the literature, as is seen by the various references we quote, starting from \cite{Karpintroduceschain}.

The advantage of chain logic is that its syntax is the same as the syntax of the classical $L_{\kappa,\kappa}$. The difference is in the concept of a model. Our purpose in this paper is to establish the mutual relationships 
between different chain logics and 
Shelah's logic. The main difficulty is that chain logic and $L^1_\kappa$ are based on a different concept of a model, chain versus the classical model. However, accepting the notion of a chain model, the logic $L^c_{\kappa, \kappa}$
has syntax, Interpolation, Undefinability of Well Order (UDWO) and an associated $EF$-game. Moreover, coming back to  Lindstr{\"o}m-style characterisations, we prove that chain logic also has it, in an appropriate sense !  Namely, at strong limits of singular cardinals of countable cofinality, chain logic
$L^c_{<\kappa, <\kappa}=\bigcup_{\lambda<\kappa} L^c_{<\lambda, <\lambda} $ is also characterised as a maximal logic with UDWO. In fact, the only known property of $L^1_\kappa$ for which it is not known if $L^c_{<\kappa, <\kappa}$ has it, or if it even makes sense to require it, is SUDWO. SUDWO is a strengthening of UDWO by an extra 
property that Shelah calls an `a posteriori' property, meaning that its definition is motivated by the proof of the characterisation theorem. We now explain some of our findings, which among other things suggest that SUDWO is strictly stronger than UDWO.

For $\kappa$ singular of countable cofinality satisfying 
$\kappa=\beth_\kappa$ (such as the first fixed point of the $\beth$-hierarchy), we prove that  the chain logic $L^c_{\kappa, \kappa}$ of \cite{Karpintroduceschain} is above $L^1_\kappa$ in the Chu order and that this is in some sense strict. The Chu order extends the one used in the characterisations above (Theorem \ref{complok} and Theorem \ref{chainbigger}). 
Our work also gives that 
 chain logic provides a solution to Problem 1.4. from \cite{Sh797}, which was one of the motivations for $L^1_\kappa$:

{\bf Problem 1.4. from \cite{Sh797}}: Suppose that $\kappa$ is a singular strong limit cardinal of countable cofinality. Is there a logic between $L_{\kappa^+, \omega}$ and $L_{\kappa^+, \kappa}$ which satisfies Interpolation?

Our Corollary \ref{solves1.4.} shows that modulo replacing $L_{\kappa^+, \kappa}= L_{\kappa, \kappa}$ (since $\kappa$ is singular, see Observation \ref{moreforless}) by a model class of $L_{\kappa, \kappa}$, the chain logic has the required properties. The logic $L^1_\kappa$, although it has Interpolation, does not solve this problem, because it is not
above $L_{\kappa^+, \omega}$ (it is above $\bigcup_{\lambda<\kappa} L_{\lambda, \omega}$).

As one of our main findings, we show that the chain logic has a  Lindstr{\"o}m-style 
characterisation as the maximal to satisfy UDWO, where UDWO and maximality are interpreted in the appropriate sense, Theorem \ref{realmax}.
We also show (Theorem \ref{incompactness}) that chain logic is not $\kappa$-compact, resolving a question about this logic which we have asked on numerous occasions.

The final part of our paper contributes to the further development of chain logic by proving the Union Lemma and identifying the chain-independent fragment of the logic, which surprisingly still has a lot of expressive power.

\subsection{Organisation, background and notation} 
The paper is organised as follows. \S\ref{intro} describes the motivation and the main result, after which it gives the mathematical background and notation. \S\ref{Chudefined} defines one of the main tools of the paper, which are Chu transforms, to be used in later sections 
for comparing various abstract logics. \S\ref{sec:chain} introduces chain logic and its basic properties, and locates it between $L_{\kappa,\omega}$ and a model class of $L_{\kappa,\kappa}$. \S\ref{sec:Shelah} introduces Shelah's logic $L^1_\kappa$. The main results on comparing  chain logic and $L^1_\kappa$ are found in \S\ref{sec:main}. The characterisation of chain logic as the maximal chain logic to satisfy UDWO
is in \S\ref{maximalitysec}. Further results about chain logic appear
in \S\ref{sec:further}. \S\ref{fin} contains 
some concluding remarks. 

Throughout the paper $\tau$ 
denotes a fixed relational language of countable size which in some contexts may be enriched with up to 
$\kappa$ constant symbols, for some fixed cardinal $\kappa$ to be specified. 
The arity of the relation symbols is always assumed to be $<\kappa$.
When we wish to discuss different languages at the same time, we may also use the term `vocabulary'  to denote each of them. The words `model of $\tau$' refer to non-empty sets equipped with the interpretation of the symbols of  $\tau$. The concept of a model, sometimes also called a `structure', refers to models of $\tau$ in which we have also interpreted the satisfaction relation of the logic we are working with.  

For simplicity, we 
 work with
a singular cardinal  $\kappa$ of countable 
cofinality and a cofinal strictly increasing sequence $\langle \kappa_n:\,n<\omega\rangle$ of cardinals with limit $\kappa$.

We study various logics 
simultaneously, so it is useful to have an abstract definition of a logic. 
From \cite{Joukoablogic}:

\begin{definition}\label{abslogic}  A {\em logic} is a triple of the form ${\mathfrak L}=(L, \models_{\mathfrak L}, S)$ where 
$\models_{\mathfrak L} \subseteq  S\times L$ and $S$ comes with a notion of isomorphism, usually understood from the context. We think of $L$ as the set or class of sentences of
${\mathfrak L}$, $S$ as a set or class of models of ${\mathfrak L}$ and of 
$\models_{\mathfrak L}$ as the satisfaction relation. The classes $L$ and $S$ can be proper classes.
If the rest is clear from the context, we often identify ${\mathfrak L}$ with $L$.
\end{definition}

To be able to use the context of \cite{Sh797}, we need to assume some 
additional properties of the logics in question.
Following Shelah's notation we define:

\begin{definition}\label{nicelogic} A logic $(L, \models_{\mathfrak L}, S)$ is {\em nice} iff it satisfies the following
requirements:
\begin{itemize}
\item for any $n$-ary relation symbol $P$ and constant symbols $c_0, \ldots c_{n-1}$ in $\tau$, 
$P[c_0, \ldots c_{n-1}]$ is a sentence in $L$,
\item $L$ is closed under negation, conjunction and disjunction,
\item for any $\varphi\in L$ and $M\in S$,
$M  \nmodels_{\mathfrak L} \varphi$ if and only if $M  \models_{\mathfrak L} \neg\varphi$,
\item $M \models_{\mathfrak L} \varphi_1\wedge \varphi_2$ iff $M \models_{\mathfrak L} \varphi_1$
and $M \models_{\mathfrak L} \varphi_2$, and similarly for disjunction,
\item for any $M\in S$, $a\in M$ and a sentence $\psi[a]\in S$ 
such that $M \models_{\mathfrak L}\psi[a]$, we have that $M\models_{\mathfrak L} (\exists x) \psi(x)$, and 
conversely, if $M\models_{\mathfrak L} (\exists x) \psi(x)$ then there is $a\in M$ such that $M \models_{\mathfrak L}\psi[a]$,
\item if $M_0$ and $M_1$ are isomorphic models of $\tau$, by some isomorphism $f$, and if both $M_0, M_1$ are in $S$, then for every $\varphi\in L$ we have $M_0  \models_{\mathfrak L} \varphi[a_0, \ldots a_{n_1}]$ iff 
$M_1  \models_{\mathfrak L} \varphi[f(a_0), \ldots f(a_{n_1})]$.
\end{itemize}
\end{definition}

{\bf Remark.} In this paper we deal with the logics of the form $L_{\lambda, \theta}$, the chain logics
$L^c_{\kappa, \kappa}$ and $L^{c,*}_{\kappa, \kappa}$, and Shelah's logic $L^1_\kappa$. They are all nice logics (a fact that is not trivial for Shelah's logic but is proved in Claim 2.7(1) of \cite{Sh797}), except that the closure under isomorphism in $L^{c,*}_{\kappa, \kappa}$ and *$L^{c}_{\kappa, \kappa}$  has to be replaced by the closure under chain isomorphism 
(Definition \ref{chainour}). We also note that our investigation will include non-standard interpretations of $\models$, see for example
equation (\ref{newtruth}), to be found a couple of paragraphs above Definition \ref{generalchainlogic}.

\begin{definition}
${\mathfrak L}$ is {\em $(\theta,\lambda)$-compact} if every set of
$\le\theta$ sentences of $L$ such that every subset of size $<\lambda$ has a model (in $S$), itself has a model (in $S$). 
${\mathfrak L}$
is {\em $\lambda$-compact} if it is  $(\theta,\lambda)$-compact for all $\theta$. 
\end{definition}

Our notation follows that of \cite{Ebbinghaus}.\footnote{Some authors denote what we call 
$\lambda$-compactness by {\em strong} $\lambda$-compactness.}

We shall often discuss logics of the form $L_{\lambda,\theta}$, 
in which sentences
are formed as in the case of first order logic, but we are allowed to use conjunctions and disjunctions of
length $<\lambda$ and strings of $<\theta$ existential quantifiers. The following is a well-known observation about such logics.

\begin{observation}\label{moreforless} If $\lambda$ is a singular cardinal, then every $L_{\lambda^+,\theta}$
sentence is expressible in an equivalent form as an $L_{\lambda,\theta}$ sentence.
\end{observation}

\begin{proof} The proof is by induction on the complexity of the sentence, where 
 $\bigvee_{i<\lambda} \varphi_i$ is translated into $\bigvee_{j<\cf(\lambda)} \bigvee_{i<\lambda_j} \varphi_i$
for an increasing sequence of cardinals $\langle \lambda_j:\,j<\cf(\lambda)\rangle$ converging to $\lambda$.
$\eop_{\ref{moreforless}}$
\end{proof}

\section{Comparison between logics and Chu transforms}\label{Chudefined}
We develop a general approach for comparing logics by using the notion of Chu 
spaces. This notion, studied in category theory (cf. \cite{BarrChu} and \cite{ChuonChu}) has proven useful in computer sciences and 
generalised logic. A simple way to understand the purpose of Chu tranforms in logic is to recognise that they really are a somewhat more sophisticated version of interpretability of one logic in anotehr. A similar notion to ours was studied in \cite[Definition 2.2]{Joukoablogic}. \footnote{ We note that in the area of cardinal invariants in set theory of the reals and topology, similar reductions were introduced by Vojt{\'a}{\v s} in \cite{Vojtas}.} The main
difference with our definition is in the density condition below. We give Observation \ref{compactnesspreservation} as a 
version of Lemma 2.3 of  \cite{Joukoablogic}.

\begin{definition}\label{def:Chu} A {\em Chu space} over a set $K$ is a triple $(A, r, X)$ where $A$ is a set of points, $X$ is a set of states and the function $r:\,A\times X\into K$ is a $K$-valued binary relation between the elements of $A$ and the elements of $X$. When $K=\{0,1\}$ we just speak of Chu spaces and $r$ becomes a 
2-valued relation.

A {\em Chu transform} between Chu spaces $(A, r, X)$ and $(A', r,' X')$ over the same set $K$ is a pair of functions
$(f,g)$ where $f:\, A\to A'$, $g:\,X'\to X$ and 
 the {\em adjointness condition} 
$r'(f(a), x'))= r(a, g(x'))$ holds.
\end{definition}

We  consider abstract logics $(L,\models, S)$ (see Definition \ref{abslogic}) as Chu spaces and we use the
following notion of Chu transforms to quasi-order abstract logics. 

\begin{definition}\label{def:lesseq} We say that $(L,\models, S)\le (L',\models', S') $ if there is a Chu transform $(f,g)$ between 
 $(L,\models, S)$ and $(L',\models', S')$ 
 such that the range of $g$ is {\em dense} in the following sense:
 \begin{itemize}
\item for every $\phi \in L$ for which there is $s\in S$ with $s\models \phi$, there is $s'\in S'$ with 
$g(s')\models \phi$.
\end{itemize}
\end{definition}

For example, any $g$ which is onto will  satisfy the density condition. 

\begin{observation}\label{partailorder} The relation $\le$ introduced in Definition \ref{def:lesseq} is reflexive
and transitive.
\end{observation}

Therefore Chu transformations provide a partial order among logics. This order preserves and reflects many properties of the logics being compared. An important one is illustrated by the following one, observed by Francesco Parente \cite{fparente} as a comment on an earlier draft of this paper. 

\begin{observation}[Parente \cite{fparente}]\label{francesco} Suppose that $(L,\models, S)$ and $(L',\models', S')$ are nice logics and that
 $(L,\models, S)\le (L',\models', S')$.
Then, for any boolean combination $\varphi$ of $\psi_0, \ldots, \psi_n$ in $L$, $f(\varphi)$ 
 is $L'$-equivalent to the corresponding boolean combination of $f(\psi_0), \ldots, f(\psi_n)$.
\end{observation}

\begin{proof} It suffices to show the claim for $\neg$ and an arbitrary $\wedge$. Let $\varphi\in L$, then for 
every $M'\in S'$ we have 
\[
M'\models' f(\neg \varphi)\iff g(M')\models (\neg \varphi)\iff 
\]
\[ \iff \mbox{not } g(M')\models \varphi
\iff \mbox{not } M'\models' f(\varphi) \iff  M'\models' \neg f(\varphi).
\]
Let $\Phi\subseteq L$ be such that $\bigwedge \Phi$ exists in $L$. Then for every $M'\in S'$,
\begin{multline*}
M'\models' f(\bigwedge \Phi)\iff g(M')\models \bigwedge \Phi\iff\mbox{ for each } \varphi\in \Phi, 
\, \,g(M')\models \varphi
\\
\iff \mbox{ for each } \varphi\in \Phi, 
\, M'\models' f(\varphi) \iff  M'\models' \bigwedge_{\varphi\in \Phi} f(\varphi).
\eop_{\ref{francesco}}
\end{multline*}
\end{proof}

Chu transforms also preserve compactness. The proof below is typical for many preservation properties of these transforms.

\begin{observation}\label{compactnesspreservation} Suppose that $(L,\models, S)\le (L',\models', S')$ and 
$(L',\models', S')$ is compact. Then so is  $(L,\models, S)$.
\end{observation}

\begin{proof} Let $(f,g)$ be the Chu transform which witnesses $(L,\models, S)\le (L',\models', S')$.
Suppose that $\Sigma\subseteq L$ is finitely satisfiable and let $\Sigma'=\{ f(\varphi): \varphi\in \Sigma\}$.
We now claim that $\Sigma'$ is finitely satisfiable. Any finite $\Gamma'\subseteq \Sigma'$ is of the
form $\{ f(\varphi): \varphi\in \Gamma\}$ for some finite $\Gamma\subseteq \Sigma$. Therefore there is $M\in S$
with $M\models \varphi$ for all $\varphi\in \Gamma$. Since $g$ is not necessarily onto, we cannot use it to obtain from $M$ an element of $S'$.

However, $\bigwedge \Gamma$ is a sentence of $L$, by  closure under conjunctions. Since 
$\models$ satisfies Tarski's definition of truth for  quantifier-free formulas,  the fact that 
$M\models \varphi$ for all $\varphi\in \Gamma$ implies  $M\models \bigwedge \Gamma$.
By the density requirement on $g$, there is $M'\in S'$ such that $g(M') \models \bigwedge \Gamma$ and hence
$M'\models' f(\bigwedge \Gamma)$. By the preservation of the logical operations by $f$, we have 
that $f(\bigwedge \Gamma)$ is $L'$-equivalent to $\bigwedge_{\varphi\in \Gamma} f(\varphi)$, so that $M'\models' f(\varphi)$
for all $\varphi\in \Gamma$ and $M'\models'\Gamma'$. Thus $\Gamma'$ is finitely satisfiable in $S'$, which by the assumption implies that there is $N'\in S'$ with $N'\models \Sigma'$. Therefore $g(N')\models \Sigma$.
$\eop_{\ref{compactnesspreservation}}$
\end{proof}

With slight modifications the proof of Observation \ref{compactnesspreservation}  goes through for the higher
degrees of compactness: 

\begin{corollary}\label{kappacompactnesspreservation} Suppose that $(L,\models, S)\le (L',\models', S')$ as witnessed by a pair $(f,g)$. Suppose
that the following conditions are satisfied:
\begin{enumerate}
\item $L, L'$ are closed under conjunctions of $<\lambda$ sentences,
\item $\models$ satisfies Tarski's definition of truth for  quantifier-free formulas, including  conjunctions and
disjunctions of size $<\lambda$,
\item $f$ preserves  conjunctions and disjunctions of size $<\lambda$, in the sense of Observation \ref{francesco}.
\end{enumerate}

Then, for any $\theta$, if $(L',\models', S')$ is $(\lambda, \theta)$-compact, so is $(L,\models, S)$.
\end{corollary}

We consider two logics $\mathfrak L$ and $\mathfrak L'$ equivalent if $\mathfrak L\le \mathfrak L'$ and 
$\mathfrak L'\le \mathfrak L$. For the purpose of studying various properties of logics, which is our main concern, once we prove that the relevant property is preserved by a Chu transform,  two equivalent logics are 
the same with respect to that property.  One such property is compactness, as we proved in Observation \ref{compactnesspreservation}.

\section{Chain logic}\label{sec:chain}
In \cite{singtrees} the authors explored  Karp's notion of a {\em chain model} and  the associated model theory of $L_{\kappa, \kappa}$ based on that concept. The definition used in  \cite{singtrees} was as follows. 

\begin{Definition}\label{chainour} A {\em chain model (of length $\mu$)} consists of a model $\mathfrak A$ and a sequence $\langle A_\alpha:\,\alpha<\mu\rangle$, called a {\em chain}. It is assumed that  $A=\bigcup_{\alpha<\mu} A_\alpha$, and that $\langle A_\alpha:\,\alpha<\mu\rangle$  is an increasing
sequence of sets\footnote{hence, automatically, models of $\tau$, since $\tau$ is relational} satisfying $|A_\alpha|<|A|$ for each $\alpha$.
A {\em chain isomorphism} between
$\mathfrak A$ and $\mathfrak B$ is an isomorphism $f:\,\mathfrak A\into \mathfrak B$
such that for all $\alpha$ the image of  $A_\alpha$ is contained in some $B_\beta$ and, conversely, the preimage of any $B_\alpha$
is contained in some $A_\beta$. If there is a chain isomorphism between $\mathfrak A$ and $\mathfrak B$ we write $\mathfrak A\cong^{c}\mathfrak  B$. \end{Definition}

Notice that the decomposition of a model $\mathfrak A$ into a chain is not required to be elementary. 
The instance that we study specifically is $\mu=\omega$, when we shall often write $\mathfrak A=(A_n)_{n<\omega}$.
The defining feature of chain models is the following modification of the
truth definition
of $L_{\lambda, \kappa}$, given by the induction on complexity of $\varphi$. For quantifier-free sentences,
the notions of $\models$ and  $\models^c$ agree, and the quantifier case is covered by the following:

\begin{equation}\label{newtruth}
(A_n)_{n<\omega}\models^c\exists \bar{x}\varphi( \bar{x})\iff\mbox {there are }n<\omega \mbox{ and } \bar{a}\in A_n^{<\kappa} \mbox{ with }
(A_n)_{n<\omega}\models \varphi[ \bar{a}],
\end{equation}
where $\bar{x}$ is a sequence of length $<\kappa$. If we restrict  our attention to chain models, the model
theory of $L_{\kappa, \kappa}$ (equivalently $L_{\kappa^+, \kappa}$ as $\kappa$ is singular) is very much like that of $L_{\omega_1,\omega}$. Karp proved several theorems about  chain logic, and further results by various authors are the
undefinability of well order UDWO (Makkai \cite{Makkaichainlogic}, \cite{Cunnighamthesis}), Craig Interpolation Theorem, Beth Definability Theorem (\cite{Cunnigham}), and an extension of Scott's analysis of countable models to chain
models of size $\kappa$ (which is the contribution
 of \cite{singtrees}).\footnote{More historical details and comparisons with the classical $L_{\kappa, \kappa}$
can be found in \cite{singtrees}.} Various kinds of chain models and their interaction are presented in \S\ref{kindsof} below.

One can use the concept of a chain model to study other logics besides $L_{\kappa, \kappa}$,  adapting the truth definition from (\ref{newtruth}) accordingly, and in fact this was Karp's approach. See \cite{Cunnigham} for a historical overview. 

\begin{definition}\label{generalchainlogic} For a logic $L$, we shall denote by $L^c$ the logic obtained from $L$ by changing the truth definition to use the notion of a chain model, as in (\ref{newtruth}). Some details of this passage might depend on the logic, in which case we specify it in the context.  We only consider
chain models of countable length.
\end{definition}

\subsection{Various kinds of chain models}\label{kindsof} As mentioned in the introduction,
Definition \ref{chainour} 
permits several modifications. We specify them in the following definition.

\begin{Definition}\label{chainournew} (1) A {\em weak chain model} consists of a model $A$ of  $\LL$ and a sequence $\langle A_n:\,n<\omega\rangle$ of subsets of  $A$, called a {\em chain}, where the only assumption that is made is that for all $n<\omega$, $A_n\subseteq A_{n+1}$ and $A=\bigcup_{n<\omega}A_n$. 

We denote such a weak chain model as $(A_n)_{n<\omega}$. If the meaning is clear from the context, we may just use {\em  ``chain model"} for ``weak chain model".

{\noindent (2)} (equivalent to Definition \ref{chainour})
A {\em chain model of strict power at most $\kappa$} is a weak chain model $(A_n)_{n<\omega}$ where for each $n$,   $|A_n|<|A|$ and $|A|\le\kappa$.  

{\noindent (3)} A  {\em proper chain model} is a weak chain model $(A_n)_{n<\omega}$ where for each $n$,  $|A_{n}|<|A_{n+1}|$.   
\end{Definition}

\begin{observation}\label{variouschainmodels} (1) Any ordinary model $\mathfrak A$ is a weak chain model, where  all elements are the same set, the universe of $\mathfrak A$. In particular, weak chain models exist in all cardinalities.

{\noindent (2)} The cardinality requirements on proper chain models mean that such models can only exist in cardinals of countable cofinality. 

{\noindent (3)} Every chain model $(A_n)_{n<\omega}$ of strict power at most $\kappa$ gives rise to a chain-isomorphic proper chain 
model of the form $(A_{k_n})_{n<\omega}$, obtained by a passing to a subsequence of 
$\langle A_n:\,n<\omega\rangle$. 
\end{observation}

We give a simple example showing that the truth in chain models depends on the choice of the chain, namely we produce two chain models $(A_n)_{n<\omega}$ and 
$(B_n)_{n<\omega}$ with $\bigcup_{n<\omega}A_n=\bigcup_{n<\omega}B_n$, and a sentence $\varphi$ true in 
 $(A_n)_{n<\omega}$ but false in $(B_n)_{n<\omega}$.

\begin{example}\label{dependenceonchain}  
Let $\LL=\{<\}$, where $<$ is a binary relation symbol. Let $\mathfrak A$ be the model of $\LL$ 
represented by the lexicographic sum of $\langle \kappa_n:\, n<\omega^\ast\rangle$, so each $\kappa_n$ is ordered as the corresponding ordinal and for $m<n<\omega$ we have that the block $\kappa_m$ is above the block 
$\kappa_n$. We can see this sum as consisting of pairs $(\alpha, n)$ where $n<\omega$ and $\alpha<\kappa_n$, ordered by $(\alpha, n)<(\beta, m)$ if $m<n$ or if $m=n$ and $\alpha<\beta$.

Let us choose a sequence $\langle \alpha_n:\, n<\omega\rangle$ so that $\alpha_n\in \kappa_n$,
hence the sequence $\langle (\alpha_n,n):\,n<\omega\rangle$ is $<$-decreasing in $\mathfrak A$.
Let $\varphi$ be the $L_{\kappa,\kappa}$ sentence
\begin{equation}
(\exists x_0) (\exists x_1) \ldots (\exists x_n)  \ldots \bigwedge_{n<\omega} (x_{n+1} < x_n).
\end{equation}
Let $A_0=\{(\alpha_n,n):\,n <\omega\}\cup \kappa_0 \times \{0\}$ and for $n>0$ let 
$A_n= \kappa_n \times \{n\}\cup \bigcup_{k<n} A_k$. For $n<\omega$ let $B_n=\bigcup_{k\le n} \kappa_k \times \{k\}$.

Then $(A_n)_{n<\omega}\models^c \varphi$ while $(B_n)_{n<\omega}\models^c \neg \varphi$.
\end{example}

\begin{convention} (1) Let
${\mathfrak M}^{c}$ stand for the class of proper chain models and let ${\mathfrak M}^{c, \ast}$ stand for the class of all weak chain models. 

{\noindent (2)} The class of ordinary models is denoted by $\mathcal M$ and the class of ordinary models of size $\lambda$ by $\mathcal M_\lambda$ .

{\noindent (3)} Let $L^c_{\lambda, \kappa}=(L_{\lambda,\kappa}, \models^c, {\mathfrak M}^{c})$ and let
$L^{c,\ast}_{\lambda, \kappa}=(L_{\lambda,\kappa}, \models^c, {\mathfrak M}^{c,\ast})$.

{\noindent (4)} We use the usual notation $L_{\lambda,\theta}$ to denote the logic $(L_{\lambda,\theta}, \models,
\mathcal M)$.
\end{convention}

\begin{observation}\label{wearenice} Modulo replacing the closure under isomorphisms in the last item of Definition \ref{nicelogic} by the closure under chain isomorphisms, the logics $L^c_{\kappa, \kappa}$ and 
$L^{c,\ast}_{\kappa, \kappa}$ are nice logics.
\end{observation}

\subsection{Comparison with $L_{\kappa, \omega}$ and $L_{\kappa, \kappa}$}
We start with the following easy observation. 

\begin{observation}[Karp] \label{comparison}

{\noindent (1)} Every formula of the logic $L_{\lambda, \omega}$ is a formula of $L^c_{\lambda, \kappa}$ and equally of  $L^{c,\ast}_{\lambda, \kappa}$ and on such formulas the notions of $\models$ and $\models^c$ agree. Formulas of $L_{\lambda, \kappa}$ are also formulas of 
$L^c_{\lambda, \kappa}$, but then the notions  $\models$ and $\models^c$ do not necessarily agree.

{\noindent (2)} Suppose that $\varphi$ is a sentence in $L_{\lambda, \omega}$. Let $\mathfrak A$ be a (weak, proper) chain model with a decomposition
$(A_n)_{n<\omega}$. Then $\mathfrak A$ is a model of $\varphi$ iff $(A_n)_{n<\omega}$ is a (weak, proper) chain
model of $\varphi$.
\end{observation}

\begin{proof} (1) By definition.

{\noindent (2)} By induction on the complexity of $\varphi$. If $\varphi$ is quantifier-free then the definition of 
$\models^c\varphi$ is
the same as that of $\models \varphi$. Suppose that $\phi\equiv(\exists x)\psi(x)$. If 
$\mathfrak A\models \varphi$ then there is a witness $x$ in some $A_n$ and hence 
$(A_n)_{n<\omega} \models^c\varphi$. The other direction is similar. 
$\eop_{\ref{comparison}}$
\end{proof}

The following Theorem \ref{lowerbound} is  the reason that  (weak) chain logic is not $\kappa$-compact, as we shall see in Corollary \ref{notcompact}.

\begin{theorem}\label{lowerbound}  
(1) $L_{\kappa, \omega}  \le L^{c,\ast}_{\kappa,\kappa}$.

{\noindent (2)} If $\kappa$ is a strong limit cardinal then $(L_{\kappa, \omega},\models, {\mathcal M}_{\ge \kappa})  \le L^{c}_{\kappa,\kappa}$, where ${\mathcal M}_{\ge\kappa}$ stands for models of size $\ge\kappa$.

{\noindent (3)} $(L^{c,\ast}_{\kappa,\kappa}, \models^c, {\mathcal M}^c_{ \kappa}) \le L^{c}_{\kappa,\kappa}$, where ${\mathcal M}^c_{\ge\kappa}$ stands for weak chain models of size $\ge\kappa$.
\end{theorem}

\begin{proof} (1) Let $f$ be the identity function and let $g((M_n)_{n<\omega})=\bigcup_{n<\omega} M_n$. Notice that $g$ is onto.
By Observation \ref{comparison} we have that $(f,g)$ is a Chu transform witnessing the announced inequality.

\smallskip

{\noindent (2)} The 
pair $(f,g)$ as above will still be a Chu transform, but it is not immediate that $g$ satisfies the density condition as it now acts only on proper chain models, so models of size 
$\kappa$ or some other singular cardinal of countable cofinality. The conclusion uses a downward Lowenheim-Skolem theorem for $L_{\kappa, \omega}$ that holds in the following form and which uses the assumption that 
$\kappa$ is a strong limit (see Theorem 3.4.1 in \cite{Dickmann}):

\begin{lemma}\label{DLS} Assume that $\kappa$ is a strong limit cardinal. Then any sentence $\varphi$ of $L_{\kappa, \omega}$ that has a model of size $\ge\kappa$ also has a model of size $\kappa$.
\end{lemma}

The lemma immediately implies the density of $g$, since any model of $L_{\kappa, \omega}$ of size $\kappa$ can be represented as a proper chain model and the two will satisfy the same sentences by Observation \ref{comparison}(2).

\smallskip

{\noindent (3)} This time we take both $f$ and $g$ to be the identity functions. The adjointness property is easily verified, but the density condition needs an argument. It will be provided by Corollary 2.5 of \cite{Cunnigham}, which is a Downward Lowenheim-Skolem theorem for the chain logic. It states that every $L_{\kappa, \kappa}$ -sentence which has a chain model, must have a strict chain model of power at most $\kappa$. 

So suppose that an $L_{\kappa, \kappa}$-sentence $\varphi$ has a chain model of $(A_n)_{n<\omega}$ of power at least $\kappa$, but this model is not necessary a strict chain model of size $\kappa$. Notice that $|A_n|$ must converge at least to
$\kappa$, therefore for any sequence of cardinals $\langle \kappa_n:\, n<\omega\rangle$ with supremum 
$\kappa$, for every $n$, there will be $A_m$ with $|A_m|\ge \kappa_n$. Hence the sentence
\[ 
\theta_n\equiv (\exists_{i<\kappa_n} x_i) \forall_{i,j<\omega} [x_i\neq x_j]
\] 
is true in the chain model $(A_n)_{n<\omega}$ and so is the conjunction $\theta\equiv \bigwedge_{n<\omega} \theta_n$. This shows that the sentence $\varphi\wedge \theta$ has a chain model and hence by the Downward Lowenheim-Skolem theorem for the chain logic, we have that $\varphi\wedge \theta$
 must have a strict chain model $(B_n)_{n<\omega}$ of power at most $\kappa$. The definition of 
 of $\theta_m$ implies that for every $n$, the model $(B_n)_{n<\omega}$ must have size at least $\kappa_m$. Putting this together, we get that $(B_n)_{n<\omega}$ has power exactly $\kappa$, and therefore 
 $(B_n)_{n<\omega}=g((B_n)_{n<\omega})$ is in the range of $g$ and a model of $\varphi$.
$\eop_{\ref{lowerbound}}$
\end{proof}

\section{Shelah's $L^1_\kappa$}\label{sec:Shelah}
In this section we recall the definition and basic properties of Shelah's logic $L^1_\kappa$ as defined  in \cite{Sh797}. Throughout this section we  assume that $\kappa=\beth_\kappa$, not necessarily singular.

Before defining $L^1_\kappa$ we define its elementary equivalence by means of a game. 
The game is called the ``borrowing game", $BG$,  because the partitions into $\omega$ many blocks given by the functions $h_n$ let us ``borrow" the time in which we have to ``repay" the promise of constructing a partial isomorphism.

\begin{definition}[\cite{Sh797}] \label{DG} Let $\beta$ be an ordinal and $\theta$ a cardinal. When we define $L^1_\kappa$ we assume  $\theta<\kappa$ and $\beta<{\theta}^+$. Let $\ma$ and $\mb$ be models of some fixed vocabulary $\tau$. \footnote{These models are typically much bigger than $\theta$ but not necessarily smaller than $\kappa$.}

The game
$$\dg^\beta_\theta(\ma,\mb)$$
is defined as follows:
\begin{enumerate}
\item Player I picks $\beta_0<\beta$, $\theta_0\le \theta$ and $A_0\in[A]^{\theta_0}$.
\item Player II picks $h_0:\,A_0\to\omega$ and a partial isomorphism  $g_0:h_0^{-1}(0)\to B$.
\item Player I picks $\beta_1<\beta_0$ and $B_1\in[B]^{\theta_0}$.
\item Player II picks $h_1:\, B_1\to\omega$ and a partial isomorphism 
$g_1\supseteq g_0$ such that $h_1^{-1}(0)\subseteq \ran(g_1)$ and $h_0^{-1}(1)\subseteq \dom(g_1)$.
\item Player I picks $\beta_2<\beta_1$ and $A_2\in[A]^{\theta_0}$.
\item Player II picks $h_2:\,A_2\to\omega$ and a partial isomorphism $g_2\supseteq g_1$ such that $h_2^{-1}(0)\subseteq \dom(g_2)$, $h_1^{-1}(1)\subseteq \ran(g_2)$ and $h_0^{-1}(2)\subseteq \dom(g_2)$.
\item ...
\item Eventually $\beta_n=0$ and the game ends. Player II wins if she can play to the end. Otherwise Player I wins.
\end{enumerate}

Let $\ma\sim_{\beta,\theta} \mb$
iff Player II has a winning strategy in $\dg^\beta_\theta(\ma,\mb)$. 

We define the relation (of $\ma$ and $\mb$) $$\ma\equiv^\beta_\theta\mb$$
to be the transitive closure of $\sim_{\beta,\theta}$. 
\end{definition}

\begin{observation}[\cite{Sh797}]\label{refinement} (1) For a fixed $\theta$ and $\beta\le\beta'<\theta^+$, the relation $\sim_{\beta',\theta}$
refines the relation  $\sim_{\beta,\theta}$, that is $\ma\sim_{\beta',\theta} \mb$ implies that $\ma\sim_{\beta,\theta} \mb$, for any models $\ma$ and $\mb$ of vocabulary $\tau$. The relation $\equiv^{\beta'}_\theta$ refines $\equiv^\beta_\theta$, that is $\ma\equiv^{\beta'}_\theta \mb$ implies that $\ma \equiv^\beta_\theta \mb$ for any models $\ma$ and $\mb$ of vocabulary $\tau$.

{\noindent (2)} If $\theta\le\theta'$ then for every $\beta<\theta^+$ the relation $\sim_{\beta,\theta'}$
refines the relation  $\sim_{\beta,\theta}$, that is $\ma\sim_{\beta,\theta'} \mb$ implies that $\ma\sim_{\beta,\theta} \mb$ for any models $\ma$ and $\mb$ of vocabulary $\tau$. The relation $\equiv^\beta_{\theta'}$ refines $\equiv^\beta_\theta$, that is $\ma\equiv^\beta_{\theta' }\mb$ implies that $\ma \equiv^\beta_\theta \mb$ for any models $\ma$ and $\mb$ of vocabulary $\tau$.

\end{observation}

\begin{proof} (1) Any play of $\dg^\beta_\theta(\ma,\mb)$ is a play of $\dg^{\beta'}_\theta(\ma,\mb)$, so the first 
conclusion follows. The second conclusion follows from the first. The statement (2) is proved similarly.
$\eop_{\ref{refinement}}$
\end{proof}

It is not known whether $\sim_{\beta,\theta}$ is an equivalence relation, although we suspect that it is not.

\begin{definition}[\cite{Sh797}]\label{loksatisfaction} If $\beta<\theta^+$, then any union of equivalence classes of 
$\equiv^\beta_\theta$ is called a {\em sentence} of $L^1_{\theta}$. Let
$$L^1_\kappa=\bigcup_{\theta<\kappa}L^1_{\theta}.$$
We define the satisfaction relation in $L^1_\kappa$ by letting $\ma\satisfies_{L^1_\kappa} K$ if $\ma\in K$.
\end{definition}

Note that $L^1_\kappa$ does not (a priori) have a syntax, we just have a semantical description. The definition of sentences is
such that every sentence $K$ is a proper class of models. The question arises how to define the basic logical operations, especially  negation. The following theorem is helpful. 

\begin{theorem}[Fact 2.4.(5) of \cite{Sh797}]\label{smallnumberofclasses} The equivalence relation
$\equiv^\beta_\theta$ has $\le \beth_{\beta+1}(\theta)$ equivalence classes.
\end{theorem}

Theorem \ref{smallnumberofclasses} in particular shows that the family of the $\equiv^\beta_\theta$ equivalence classes is a set, and we have defined sentences as unions of subfamilies of elements of such sets. This allows to define the logical operations $\neg$ and $\vee$ (and hence all the other logical operations), as follows:

\begin{definition}[\cite{Sh797}] \label{logoperation} Suppose that $K$ is a sentence of $L^1_\kappa$, so there are the least 
$\theta<\kappa$ and the least $\beta<\theta^+$ such that $K$ is a union of equivalence classes of 
$\equiv^\beta_\theta$. Let $\EE^\beta_\theta$ denote the set of all equivalence classes of 
$\equiv^\beta_\theta$. We define
\[
\neg K=\bigcup \{e\in \EE^\beta_\theta:\, e\notin K\}.
\]
For $1\le n<\omega$ and sentences $K_i \,(i<n)$, we find the least $\theta<\kappa$ and least $\beta<\theta^+$ such that each $K_i$ is a union of equivalence classes of 
$\equiv^\beta_\theta$ (such $\theta$, $\beta$ exist by Observation \ref{loksatisfaction}). We define
\[
\bigvee_{i<n} K=\bigcup K_i=\{e\in \EE^\beta_\theta:\, (\exists i<n)\,e\in K_i\}.
\]
\end{definition}

If $\kappa=\beth_\kappa$  is singular, $L^1_\kappa$ has Craig Interpolation and a Lindstr\"om-type characterisation involving a strong undefinability of well-order called SUDWO. Moreover, $L^1_\kappa$ is the classically maximal logic stronger than $\bigcup_{\lambda<\kappa} L_{\lambda,\omega}$ which satisfies SUDWO. These facts are all proved in \cite{Sh797}.

\section{Comparison between $L^1_\kappa$ and the chain logic}\label{sec:main}
Based on the EF games for chain logic (see \cite{singtrees}), we can define:

\begin{definition}\label{chainbeta} Suppose that $\ma=(C_n)_{n<\omega}$ and $\mb=(D_n)_{n<\omega}$ are chain models.
For a sequence 
$\xb$ of length $<\kappa$ we say that it is  \emph{bounded} if it is contained in some $C_n$ or $D_n$. For simplicity, assume $(\bigcup_nC_n)\cap(\bigcup_n D_n)=\emptyset$.

Let $\beta<\kappa^+$.
Define the game $EF^c_{\kappa,\beta}$ in which Player I
plays a sequence of bounded sequences $\xb_i$, each of length $<\kappa$, and ordinals 
$\beta_0>\beta_1>...$, where 
$\beta_0<\beta$ (so the game is finite) and II responds by an increasing sequence $g_i$ of partial chain isomorphisms
of cardinality $<\kappa$, 
including $\xb_i$ in its domain if $\xb_i$ is in $\ma$ or in its range if $\xb_i$ is in $\mb$. The first player that cannot move loses. 
We write $\ma\sim^{c}_{\kappa, \beta}\mb$  if Player II has a winning strategy in $EF^c_{\kappa,\beta}$.

The global EF game $EF^c_{\kappa,\infty}$ is defined so that at each step $i<\omega$ Player I
plays a bounded sequence $\xb_i$ of length $<\kappa$ in either $\ma$ or $\mb$. Player II responds by an a partial chain isomorphism $g_i$
including $\xb_i$ in its domain if $\xb_i$ is in $\ma$ or in its range if $\xb_i$ is in $\mb$ and so that 
$\bigcup_{j<i} g_j\subseteq g_i$. The first player that cannot move loses and II wins if the game lasts $\omega$ steps. We write $\ma\sim^{c}_{\kappa, \infty}\mb$  if Player II has a winning strategy in $EF^c_{\kappa,\infty}$.

\end{definition}

It is worth noting that the relations $\sim^c_{\kappa, \beta}$ and  $\sim^c_{\kappa, \infty}$ are equivalence relations.

\begin{lemma}\label{chainEF} Two chain models $\ma=(C_n)_{n<\omega}$ and $\mb=(D_n)_{n<\omega}$ are in the relation $\ma\sim^c_{\kappa, \beta}\mb$ iff they satisfy the same $L^c_{\infty, \kappa}$-sentences of quantifier rank $<\beta$. 
\end{lemma}

\begin{proof}
This is a standard proof, going back to \cite{MR0209132}.
$\eop_{\ref{chainEF}}$
\end{proof}

\begin{observation}\label{localversusglobal} If Player II has a winning strategy in $EF^c_{\kappa,\infty}$ then
Player II has a winning strategy in $EF^c_{\kappa,\beta}$ for any $\beta<\kappa^+$. In other words,
$\sim^c_{\kappa, \infty}\subseteq \bigcap_{\beta<\kappa^+}\sim^c_{\kappa, \beta}$
\end{observation}

\begin{proof} If $\beta<\kappa^+$ is given, we let Player II in
$EF^c_{\kappa,\beta}$ follow the same strategy as in $EF^c_{\kappa,\infty}$, which will be winning in 
$EF^c_{\kappa,\beta}$ as well. 
$\eop_{\ref{localversusglobal}}$
\end{proof}

For a converse of Observation \ref{localversusglobal} in the case of proper chain models, see
Lemma \ref{Scottwo}.

\begin{lemma}\label{TFAEequivalence} Let $\kappa$ be a singular strong limit cardinal of cofinality $\omega$ and $\mathcal K$ a class 
 of chain models of a fixed vocabulary $\tau$. Then the following are equivalent:

\begin{enumerate}
\item $\mathcal K$ is definable by a formula of $L^{c}_{\infty,\kappa}$ with the quantifier rank $< \beta$.
\item The model class $\mathcal K$ is closed under $\sim^c_{\kappa, \beta}$.
\item $\mathcal K$ is definable by a formula of $L^{c}_{\lambda,\kappa}$ with the quantifier rank $< \beta$, where $\lambda=\max(\kappa,(\beth_{\beta+2}(|\tau|)^+)$.
\end{enumerate}
\end{lemma}

Recall that by Definition \ref{generalchainlogic}, the notation $L^{c}_{\infty,\kappa}$ means that we are taking formulas of $L_{\lambda, \kappa}$ but considering them in chain models (of countable length). 

\begin{proof} 
The direction $(1)\implies (2)$ 
follows from Lemma~\ref{chainEF}.

To prove that $(2)\implies (3)$, similarly to Lemma 7.55 in \cite{joukobookgames}, we show by induction on 
$\beta$ that the number of formulas with  quantifier rank $<\beta$ over a set $X$ of free variables has size $\le \beth_{\beta+1}(|X|+ |\tau|)$. Each $\sim^c_{\kappa, \beta}$-equivalence class is determined by the set of sentences with the quantifier rank $<\beta$ which are true in one or any model in the class. Therefore, there are $\le \beth_{\beta+2}(|\tau|)$ such classes $e$. Choose for each $e$ such that $\mathcal K\cap e\neq \emptyset$ a representative $\ma_e\in \mathcal K$. Hence $\mathcal K$ is definable by the formula
\[
\bigvee_{e\cap \KK\neq \empty}
\bigwedge \{\varphi:\, \mbox{quantifier rank}(\varphi)\le\beta\,\,\&\,\, \ma_e\models^c\varphi\}.
\]

The direction $(3)\implies (1)$ follows from the definition.
$\eop_{\ref{TFAEequivalence}}$
\end{proof}

\begin{definition}\label{Kstar} For a class $K$ of models of $\tau$, let $\CC(K)$ stand for the class of all
possible weak chain representations of models in $K$.
\end{definition}

\begin{lemma}\label{betaequiv} If $\kappa$ is as above, $\beta<\kappa$ and $(C_n)_{n<\omega}$
and $(D_n)_{n<\omega}$ are 
$L^{c,2\cdot\beta}_{\kappa,\kappa}$-equivalent (i.e. equivalent for $L^{c}_{\kappa,\kappa}$-sentences of quantifier rank $<2\cdot\beta$), then $\bigcup_{n<\omega} C_n \equiv^\beta_\theta
\bigcup_{n<\omega} D_n$ (in the sense of Definition \ref{DG}) for any $\theta<\kappa$ for which 
$\beta<\theta^+$.

Consequently, suppose that $K$ is a class of models of $L^1_\kappa$ closed under $\equiv^\beta_\theta$. Then the class $\CC(K)$ is closed under 
$\sim^c_{\kappa,2\cdot\beta}$ and hence definable by an $L^c_{\kappa, \kappa}$ sentence of quantifier rank $<2\cdot \beta$.
\end{lemma}

\begin{proof} In fact we prove that $\ma\sim_{\beta,\theta}\mb$, by providing a winning strategy to II in the
game $\dg^\beta_\theta(\ma,\mb)$. By the assumption  
$\ma\sim^c_{\kappa,2\cdot\beta}\mb$, we  fix a winning strategy $\sigma$ for II in $EF^c_{\kappa,2\cdot\beta}$. The $n$-th move of II in the game $\dg^\beta_\theta(\ma,\mb)$ is supposed to produce a pair of functions $(h_n, g_n)$.  To define $h_n$ for $n$ even we let for $a\in A_{n}$:
$$h_{n}(a)= \min\{k : a\in C_k\}$$ and for $n$ odd, for  $b\in B_{n}$ we let
$$h_{n}(b)= \min\{k : b\in D_k\}.$$
That is, player II simply uses the existing chain partitions of $\ma$ and $\mb$. 

To find $g_n$ Player II uses 
$\sigma$ twice. We first assume that $n=2k$ is even. Assume we have defined $g_{n-1}$
(stipulating $g_{-1}=\emptyset$) so that the requirements for II in the game $\dg^\beta_\theta(\ma,\mb)$
are satisfied.
 Then II extends $g_{n-1}$ to a partial isomorphism whose domain includes $h_0^{-1}(2k)\cup\ldots \cup  h_{2k}^{-1}(0)$ by using $\sigma$ in the play of
$EF^c_{\kappa,\beta}$ where $\sigma$ has given 
$g_{n-1}$ in the $2(n-1)$-th step of II and I has played  $h_0^{-1}(2k)\cup\ldots \cup  h_{2k}^{-1}(0)\setminus \dom(g^{n-1})$ in the step $2(n-1)+1$. In this way we extend $g_{n-1}$ to $g'_n$. Then we continue that play of $EF^c_{\kappa,\beta}$ by letting I
respond by $g_0^{-1}(2k)\cup\ldots \cup  g_{2k}^{-1}(0)\setminus \dom(g'_n)$ 
and letting II use $\sigma$
to extend $g'_n$ to $g_n$.
The strategy for $n$ odd is similar. 
 
Now suppose that $K$ is closed under $\equiv^\beta_\theta$ and that 
$\ma=(C_n)_{n<\omega}$ is in $\CC(K)$, while $\mb=(D_n)_{n<\omega}$ is $\equiv^{c,2\cdot\beta}_{\kappa, \kappa}$-equivalent to $\ma$. Then, on the one hand, $\bigcup\ma\in K$, by the definition of $\CC(K)$ and on the other hand,
$\bigcup\ma \equiv^\beta_\theta \bigcup\mb$ by the first part of the lemma. Hence $\mb\in K$, as $K$ is closed under 
$\equiv^\beta_\theta$-equivalence and therefore $(D_n)_{n<\omega}\in \CC(K)$ by the definition of $\CC(K)$.
$\eop_{\ref{betaequiv}}$
\end{proof}

For the rest of this section, in addition to the running assumption that $\kappa$ is a singular cardinal of cofinality 
$\omega$, we  assume that 
$\kappa=\beth_\kappa$.

\begin{theorem}\label{complok} For $\kappa$ as above and $\tau$  a fixed vocabulary, we have:
$$(L^1_\kappa, \models_{L^1_\kappa}, \KK) \le (L^{c,\ast}_{\kappa, \kappa}, \models^c, \KK^c),$$
where $\KK$ is the class of all models of $\tau$ and $\KK^c$ the class of all weak chain models of $\tau$.
\end{theorem}

\begin{proof} We define a Chu transform $(f,g)$ which will witness the desired relation.

To define $f$, suppose that $K$ is a sentence of $L^1_\kappa$. We shall define $f(K)$ to be a certain sentence of $L^{c}_{\kappa,\kappa}$, as we now describe. 

By the definition of what a sentence is in $L^1_\kappa$ (Definition \ref{loksatisfaction}), there are 
$\beta<\theta^+<\kappa$ such that $K$ is the union of certain
equivalence classes of $\ma\equiv^\beta_\theta\mb$ and in particular $K$ is closed under the relation 
$\sim_{\beta, \theta}$. To associate a sentence of $L^{c}_{\kappa,\kappa}$ to $K$,  we prove that $\CC(K)$ is definable in $L^{c}_{\kappa,\kappa}$. By Lemma \ref{TFAEequivalence}, since $\kappa=\beth_\kappa$, it is suffices to show that $\CC(K)$ is closed under 
$\sim^c_{\kappa,\beta}$. This follows from Lemma \ref{betaequiv}. Hence now we have an 
$L^{c}_{\kappa,\kappa}$-sentence which defines $\CC(K)$ and we choose any such sentence to be $f(K)$.

To define $g$, suppose that $(C_n)_{n<\omega}$ is a weak chain model of $\tau$. Then we let $g((C_n)_{n<\omega})=\bigcup_{n<\omega} C_n$. 

Now that we have defined $f$ and $g$, it remains to prove that the adjointness and the density condition are satisfied. 
To prove that $(f,g)$ satisfies the adjointness condition, suppose that for some chain model $(C_n)_{n<\omega}$ and an $L^1_\kappa$ sentence
$K$ we have that $\ma=\bigcup_{n<\omega} C_n\models_{L^1_\kappa} K$. This means that for some
relevant $\beta, \theta$, $K$ is a union of $\equiv^\beta_\theta$-equivalence classes and that $\ma\in K$. By Lemma \ref{betaequiv}, $\CC(K)$ is closed under $\equiv_{L^{c,\beta}_{\kappa, \kappa}}$ and by Lemma \ref{TFAEequivalence}, it is $L^c_{\kappa, \kappa}$-definable by a sentence $\varphi$ of quantifier rank $<\beta$. Moreover, we have chosen
$f(K)$ to be some $L^c_{\kappa, \kappa}$-equivalent representation of $\varphi$. Since 
$(C_n)_{n<\omega}\in \CC(K)$, we have that $(C_n)_{n<\omega}\models^c \varphi$ and hence
$(C_n)_{n<\omega}\models^c f(K)$. In the other direction, suppose that $(C_n)_{n<\omega}\models^c f(K)$ for some $K$. We can find the corresponding $\beta, \theta$ and $\varphi$ as in the previous argulent, so that we now have that
 $(C_n)_{n<\omega}\models^c \varphi$. Then by the choice of $\varphi$ it follows that 
$\bigcup_{n<\omega} C_n\models_{L^1_\kappa} K$.



The density condition is easy to prove since $g$ is onto: any model is the union of a weak chain model.
$\eop_{\ref{complok}}$
\end{proof}

We now show that $L^{c,\ast}_{\kappa, \kappa}$ is strictly stronger than $L^1_\kappa$ in Shelah's sense. 
Namely, Shelah \cite{Sh797} only considers the classically studied order between logics, where 
$\LL_1\le \LL_2$ means that for every $\tau$ and $\varphi\in \LL_1(\tau)$
there is $\psi\in \LL_2(\tau)$ such that  for every model,  $M\models_{\LL_1} \varphi$ iff $M\models_{\LL_2} \psi$. In this definition, the models are fixed, so in the language of Chu transforms, $g$ is the identity,
whereas we wish to compare two logics for which the notions of a model are different, such as  chain logic and $L^1_\kappa$. Then   non-trivial Chu transforms are needed.  In the spirit of the classical definition, our definition of a Chu transform  says that $g$ is `practically' the identity. We shall use similar technology to show that $L^{c,\ast}_{\kappa, \kappa}$ is not in this `Chu identitary' sense weaker than $L_{\kappa,\kappa}$, in contrast to our results in
Theorem \ref{PCtheorem}, which say that $L^{c,\ast}_{\kappa, \kappa}$ is Chu below a model class of 
$L_{\kappa,\kappa}$.

\begin{definition}\label{} Suppose that $(f,g)$ is a Chu transform between logics $\LL_1=(L_1, \models_1, S_1)$ and $\LL_2=(L_2, \models_2, S_2)$ is in the same vocabulary $\tau$, where the elements of $S_1$ are chain models and the elements of $S_2$
are ordinary models. Let $\sim$ be an equivalence relation among the elements of $S_2$.

We say that $(f,g)$ is an {\em $\sim$-identitary} Chu transform if $g$ is onto and, moreover, for every chain model
$(A_n)_{n<\omega}$ in $S_1$ and $\ma\in S_2$ with $\ma= \bigcup_{n<\omega} A_n$, 
there is $\mb\in S_2$ such that 
\[
g(\mb)=\langle A_n:\,n<\omega\rangle\mbox{ and } \mb\sim\ma.
\]
We write $\LL_1\le_i^\sim \LL_2$ if there is an $\sim$-identitary Chu transform  $(f,g)$ between  $\LL_1$ and $\LL_2$.
\end{definition}

\begin{theorem}\label{chainbigger} Suppose that $\kappa$ is a singular cardinal of countable cofinality, $\tau$ is a fixed vocabulary containing a binary relation symbol $<$, $\KK$ is the class of all models of $\tau$ and $\KK^c$ the class of all weak chain models of $\tau$.

{\noindent (1)} Let $\equiv$ stand for the elementary equivalence in $L_{\kappa,\kappa}$. Then we have:
$$ (L^{c,\ast}_{\kappa, \kappa}, \models^c, \KK^c) \nleq_i^\equiv  (L_{\kappa,\kappa}, \models, \KK). $$

{\noindent (2)} Suppose that in addition $\kappa=\beth_\kappa$ and for $\beta<\theta^+<\kappa$
let $\equiv^\beta_\theta$ be the relation defined in Definition \ref{DG}. For $\ma, \mb\in\KK$ define
 $\ma \equiv^1_\kappa\mb$ iff for all $\theta<\kappa$ and all $\beta <\theta^+$ we have
 $\ma \equiv^\beta_\theta\mb$. 
 
 Then 
 $$ (L^{c,\ast}_{\kappa, \kappa}, \models^c, \KK^c) \nleq_i^{\equiv^1_\kappa}  (L^1_\kappa, \models_{L^1_\kappa}, \KK). $$
\end{theorem}

\begin{proof} 
Let $\varphi$ be
the $L_{\omega_1, \omega_1}$ sentence which says that $<$ is a well-order. In particular, this is an
$L^{c,\ast}_{\kappa, \kappa}$-sentence. Similarly to the proof in Example \ref{dependenceonchain}, let us define an ordinary model
$\ma$ of $\neg\varphi$ as the lexicographic sum of $\omega^\ast$ many copies of $\kappa$, i.e. an infinite decreasing sequence of copies of $\kappa$. Then $\ma\models \neg \varphi$. We now define 
a chain model $(A_n)_{n<\omega}$ such that each $A_n$ consists of the first $n$
copies of $\kappa$ within $\ma$. Therefore $(A_n)_{n<\omega}\models^c\varphi$. Next, we define an $\omega$-sequence $\bar{a}$ by taking exactly one point from each copy of $\kappa$ in $\ma$, so $\bar{a}$ is an infinite decreasing $<$-sequence in $\ma$. Finally, let $B_n=A_n\cup \ran(\bar{a})$, for each $n$. Then $(B_n)_{n<\omega}\models^c\neg \varphi$ and 
$\bigcup_{n<\omega}B_n=\bigcup_{n<\omega}A_n=\ma$. We shall use these models in both parts of the proof.

\smallskip

{\noindent (1)} 
By the definition of an identitary Chu transform, there are $\mathfrak C$ and $\mathfrak D$ in $\KK$ such that:
\begin{itemize}
\item $g(\mathfrak C)= (A_n)_{n<\omega}$,
\item $g(\mathfrak D)= (B_n)_{n<\omega}$ and 
\item $\mathfrak C\equiv \mathfrak D\equiv \ma$.
\end{itemize}
By the defining property of elementary equivalence, we have that either $\mathfrak C,\mathfrak D \models f(\varphi)$ or $\mathfrak C,\mathfrak D \models \neg f(\varphi)$. However, the adjointness condition implies that $\mathfrak C \models f(\varphi)$, since $(A_n)_{n<\omega} \models^c \varphi$, while
$\mathfrak D\models
 \neg f(\varphi)$, since $(B_n)_{n<\omega} \models^c \neg \varphi$. A contradiction.

\smallskip

{\noindent (2)} 
Consider $f(\varphi)$. 
It is an $L^1_\kappa$-sentence, so there are $\beta<\theta^+<\kappa$ such that 
$f(\varphi)$ is a union of $\equiv^{\beta}_\theta$-equivalence classes. By the definition of an identitary Chu transform with respect to $\equiv^1_\kappa$, there are $\mathfrak C$ and $\mathfrak D$ in $S_2$ such that:
\begin{itemize}
\item $g(\mathfrak C)= (A_n)_{n<\omega}$,
\item $g(\mathfrak D)= (B_n)_{n<\omega}$,
\item $\mathfrak C \equiv^{\beta}_\theta \ma \equiv^{\beta}_\theta
\mathfrak D$.
\end{itemize}
In particular, $\mathfrak C\equiv^{\beta}_\theta \mathfrak D$, which means by the properties of 
$f(\varphi)$, that either $\mathfrak C,\mathfrak D \models_{L^1_\kappa} f(\varphi)$ or $\mathfrak C,\mathfrak D \models_{L^1_\kappa} \neg f(\varphi)$. However, the adjointness condition implies that $\mathfrak C
\models_{L^1_\kappa} f(\varphi)$, since $(A_n)_{n<\omega} \models^c \varphi$, while
$\mathfrak D
\models_{L^1_\kappa} \neg f(\varphi)$, since $(B_n)_{n<\omega} \models^c \varphi$. A contradiction.
$\eop_{\ref{chainbigger}}$
\end{proof}

\begin{example}\label{example} In Theorem \ref{PCtheorem} (2) we show that  $ (L^{c,\ast}_{\kappa, \kappa}, \models^c, \KK^c) \leq_i^\fallingdotseq \LL_2$ for a model class  $\LL_2$ of $L_{\kappa, \kappa}$ and
an equivalence relation $\fallingdotseq$ in which all models with the same  $\tau$-reduct are equivalent.
\end{example}

\section{Maximality of the Chain Logic}\label{maximalitysec}
Chain logic satisfies the undefinability of well order, by a theorem Theorem~\ref{Makkaiudwo} of Makkai which we recall here \footnote{We note that Makkai published an errata \cite{ErrataMakkai} to this theorem. A corrected proof was given by Cunnigham in \cite{Cunnighamthesis}, where she explains the situation but still credits the theorem to Makkai.}. In this section we show that the property $\udwo_\kappa$, defined in Definition~\ref{romu}, in a sense characterises the chain logic $L^c_{\kappa, \kappa}$, as per Theorem~\ref{maxudwo}.

\begin{theorem}[Makkai \cite{Makkaichainlogic}\label{th:udwo} Th. 4.1; Cunningham, Th. 3.4.2. of \cite{Cunnighamthesis}]\label{Makkaiudwo} Suppose that $\tau$ is a vocabulary containing $\{<, U\}$ where $<$ is a binary relation symbol and $U$ is a unary relation symbol. Further suppose that $\varphi$ is an $L_{\kappa, \kappa}$-sentence in $\tau$ such that for each $\alpha<\kappa^+$ there is a chain model
$\mb_\alpha$ of $\varphi$ satisfying that $(U^{\mb_\alpha}, <^{\mb_\alpha})$ is a well order of type $\ge\alpha$.

Then there is a strict chain model $\mb$ of $\varphi$ having cardinality $\le\kappa$ and such that $(U^{\mb}, <^{\mb})$ contains a copy of the rationals.
\end{theorem}

We abstract this property of $L^c_{\kappa,\kappa}$ to arbitrary logics on chain models:

\begin{definition}\label{romu} A {\em logic on chain models} is a logic of the form
$(L,\models_{\mathcal{L},},S)$, where $S $ is a class of weak chain models. 

{\noindent (1)} A logic on chain models is said to satisfy the property $\udwo_\kappa$ if every sentence in $L$ with a distinguished binary predicate $<$ such that for all $\alpha<\kappa^+$ there is chain model of $\varphi$ in $S$ in which $<$ is well-ordered in type $\ge\alpha$, has a proper chain model of cardinality $\le\kappa$ in $S$ in which $<$ is non-well-ordered.

{\noindent (2)} A logic on chain models is said to satisfy $\udwo_{<\kappa}$ 
if for every sentence $\varphi$ in $L$ with a distinguished binary predicate $<$,  there is $\lambda<\kappa$ such that if for all $\alpha<\lambda^+$ there is chain model of $\varphi$ in $S$ in which $<$ is well-ordered in type $\ge\alpha$, then $\varphi$ has a proper chain model of cardinality $\le\lambda$ in $S$ in which $<$ is non-well-ordered.
\end{definition}

Let us define

\begin{definition}\label{limitchain} Suppose that $\kappa$ is a limit of singular cardinals $\lambda$ with
$\cf(\lambda)=\omega$. Then we define $L^c_{<\kappa, <\kappa}=\bigcup_{\lambda<\kappa, \cf(\lambda)=
\omega} L^c_{<\lambda, <\lambda}$.
\end{definition}

We then have the following corollary of Theorem \ref{th:udwo}:

\begin{corollary}\label{lessudwo} Suppose that $\kappa$ is a limit of singular cardinals $\lambda$ with
$\cf(\lambda)=\omega$. Then $L^c_{<\kappa, <\kappa}$ has $\udwo_{<\kappa}$.
\end{corollary}

It turns out that for $\kappa$ as in Corollary \ref{lessudwo} which is in addition a strong limit, $L^c_{<\kappa, <\kappa}$ is a maximal logic
on chain models having $\udwo_{<\kappa}$. This section is devoted to the proof of that fact, which is given in
Theorem \ref{realmax}. The main theorem leading to the maximality result is the following Theorem \ref{maxudwo}, which can be seen as a quasi-maximality result for chain logic $L^c_{\kappa, \kappa}$.

\begin{theorem}\label{maxudwo} Let $\kappa$ be a singular cardinal of countable cofinality, $\tau$ a fixed vocabulary and $\KK^c$ the class of all weak chain models of $\tau$. 
Suppose that $\LL^\ast$ is a logic on chain models such that:
\begin{itemize}
\item $L^c_{\kappa, \kappa}\le\LL^\ast$ and
\item $(\LL^\ast,\models^c, \KK^c)$ satisfies $\udwo_\kappa$.
\end{itemize}
Then for every sentence $\varphi$ of $\LL^\ast$ 
there is $\alpha<\kappa^+$ such that the class of models of $\varphi$ is 
closed under $\sim^c_{\kappa,\alpha}$. Hence $\varphi$ is definable in $L^c_{\lambda,\kappa}$, where $\lambda=\max(\kappa,(\beth_{\alpha+2}(|\tau|)^+)$.
\end{theorem}

\medskip

The last sentence in the theorem follows from the rest by Lemma \ref{TFAEequivalence}. To prove the theorem, we shall need some further lemmas. The one is the crucial one and the one where the singularity of $\kappa$ is really essential. It provides a converse to Observation \ref{localversusglobal} in the case of proper chain models.

A logic is said to \emph{relativize} if for every sentence $\varphi$ of the logic with vocabulary $\tau$ and every unary predicate $P\notin\tau$ there is a sentence $\psi$ in the logic such that $\psi$ holds in a model with vocabulary $\tau\cup\{P\}$ if and only if the relativization of the model to $P$ (the ``$P$-part '' of the model) satisfies $\varphi$. Our chain logics satisfies this property but one has to bear in mind that the relativization of a proper chain model need not be a proper chain model. For details concerning relativization we refer to \cite{BarwiseFeferman}.
\begin{lemma}\label{Scottwo} Let $\kappa$ be a singular cardinal of countable cofinality and 
suppose that $(M_n)_{n<\omega}$ and  $(N_n)_{n<\omega}$ are weak chain models of the same vocabulary such that $|M_n|<\kappa$ and $|N_n|<\kappa$ for all $n$. If player II has a
winning strategy in $EF^c_{\kappa,\infty}((M_n)_{n<\omega},(N_n)_{n<\omega})$, then $(M_n)_{n<\omega}
\cong^c (N_n)_{n<\omega}$.
\end{lemma}

\begin{proof} We can choose an increasing sequence $\langle \kappa_n:\, n<\omega\rangle$ of cardinals such that
$\kappa=\sup_{n\in \omega} \kappa_n$ and so that
$|M_n|, |N_n| \le\kappa_n$. We play $EF^c_{\kappa}(M,N)$ so that Player I starts by playing a sequence covering $M_0$, then one covering $N_0$, then $M_1$, then $N_1$, etc. Player II plays according to the winning strategy. In the end we have a chain isomorphism from $(M_n)_{n<\omega}$ to $(N_n)_{n<\omega}$.
$\eop_{\ref{Scottwo}}$
\end{proof}

We now prove two lemmas that work for weak as well as proper chain models. Then we formulate the same lemmas appropriately for chain models satisfying extra conditions. The proofs are quite standard and based on \cite{MR0209132}.

\begin{lemma}\label{dd}
The following conditions are equivalent:
\begin{enumerate}
\item Player II has a winning strategy in the game 
$EF^c_{\kappa,\beta}((M_n)_{n<\omega},(N_n)_{n<\omega})$,
\item There are sets $\langle I_\xi:\,\xi<\beta, n<\omega\rangle $ such that:
\begin{enumerate}
\item If $\xi<\zeta<\beta$, then $\emptyset\ne I_\zeta\subseteq I_\xi$.
\item if $f\in I_\xi$, then $f=\langle(a_\zeta,b_\zeta): \zeta<\kappa_n\rangle$ is a partial chain isomorphism  from $(M_n)$ to $(N_n)$, 
\item if $f\in I_{\xi}$, then for all $\eta<\xi$ and all bounded $A\subseteq \bigcup_{n<\omega} M_n$ of cardinality $<\kappa$ there is $g\in I_\eta$ such that $f\subseteq g$ and $A\subseteq\dom(g)$.
\item If $f\in I_{\xi}$, then for all $\eta<\xi$ and all bounded $B\subseteq 
\bigcup_{n<\omega} N_n$ of cardinality $<\kappa$  there is $g\in I_\eta$  for some $k<\omega$  such that $f\subseteq g$ and $B\subseteq\ran(g)$.
\end{enumerate}
\end{enumerate}
\end{lemma}

\begin{proof}
During the game $EF^c_{\kappa,\beta}$ Player I plays, among other things, some elements of $\beta$ in descending order. Let us call those move \emph{rank moves}. If (1) holds then we let $I_\xi$ be the set of partial chain isomorphisms  from $(M_n)$ to $(N_n)$ that arise in the game when II is playing her winning strategy and Player I has played $\xi$ as his rank-move. Conversely, suppose such sets $I_\xi$ exist. The strategy of II is to make sure that the partial functions she plays after Player I plays $\xi$ as his rank move, are in $I_\xi$. More exactly, suppose she has played a partial chain isomorphism $f$  in $I_\xi$ and then Player I plays a bounded set $A$ of cardinality $<\kappa$ as well as $\eta<\xi$ as the rank move. Player II uses (c) above to find $g\in I_\eta$ such that $f\subseteq g$ and $A\subseteq\dom(g)$. She plays this $g$ in the game.
$\eop_{\ref{d}}$
\end{proof}

\begin{lemma}\label{ee}
The following conditions are equivalent:
\begin{enumerate}

\item Player II has a winning strategy in the game $EF^c_{\kappa,\infty}((M_n)_{n<\omega},(N_n)_{n<\omega})$.
\item There is a set $J$ such that
\begin{enumerate}
\item $J\ne\emptyset$.
\item if $f\in J$, then $f$ is a partial chain isomorphism  of cardinality $<\kappa$  from $(M_n)$ to $(N_n)$,
\item if $f\in J$, then for all bounded $A\subseteq \bigcup_{n<\omega} M$ of cardinality $<\kappa$ there is $g\in J$ such that $f\subseteq g$ and $A\subseteq\dom(g)$.
\item If $f\in J$, then for  all bounded $B\subseteq \bigcup_{n<\omega} N_n$  of cardinality $<\kappa$  there is $g\in J$ such that $f\subseteq g$ and $B\subseteq\ran(g)$.
\end{enumerate}
\end{enumerate}
\end{lemma}

\begin{proof}
If (1) holds then we let $J$ be the set of partial chain isomorphisms  of cardinality $<\kappa$  from $(M_n)$ to $(N_n)$ that arise in the game when II is playing her winning strategy. Conversely, suppose such sets $J$ exist. The strategy of II is to make sure that the partial functions she plays  are in $J$. More exactly, suppose she has played a partial chain isomorphism $f$ of cardinality $<\kappa$ in $J$ and then Player I plays a bounded set $A$ of cardinality $<\kappa$. Player II uses (c) above to find $g\in J$ such that $f\subseteq g$ and $A\subseteq\dom(g)$. She plays this $g$ in the game.
$\eop_{\ref{e}}$
\end{proof}

For singular $\kappa$ we can be more specific:

\begin{lemma}\label{d}Let $\kappa$ be a singular cardinal of  cofinality $\omega$ and $\kappa=\sup_n\kappa_n$, where $\kappa_n<\kappa_{n+1}$.
The following conditions are equivalent:
\begin{enumerate}
\item Player II has a winning strategy in the game 
$EF^c_{\kappa,\beta}((M_n)_{n<\omega},(N_n)_{n<\omega})$,
\item There are sets $\langle I^n_\xi:\,\xi<\beta, n<\omega\rangle $ such that:
\begin{enumerate}
\item If $\xi<\zeta<\beta$, then $\emptyset\ne I^n_\zeta\subseteq I^n_\xi$.
\item if $f\in I^n_\xi$, then $f=\langle(a_\zeta,b_\zeta): \zeta<\kappa_n\rangle$ is a partial chain isomorphism  from $(M_n)$ to $(N_n)$, 
\item if $f\in I^n_{\xi}$, then for all $\eta<\xi$ and all bounded $A\subseteq \bigcup_{n<\omega} M_n$ of cardinality $\le\kappa_{m}$, $m>n$, there is $g\in I^{m}_\eta$ such that $f\subseteq g$ and $A\subseteq\dom(g)$.
\item If $f\in I^n_{\xi}$, then for all $\eta<\xi$ and all bounded $B\subseteq 
\bigcup_{n<\omega} N_n$ of cardinality $\le\kappa_{m}$, $m>n$,  there is $g\in I^m_\eta$  such that $f\subseteq g$ and $B\subseteq\ran(g)$.
\end{enumerate}
\end{enumerate}
\end{lemma}

\begin{proof}
 If (1) holds then we let $I^n_\xi$ be the set of partial chain isomorphisms  of cardinality $\le\kappa_n$  from $(M_n)$ to $(N_n)$ which are subsets of cardinality $\le\kappa_n$ of partial functions that arise in the game when II is playing her winning strategy and Player I has played $\xi$ as his rank-move. Conversely, suppose such sets $I^n_\xi$ exist. The strategy of II is to make sure that the partial functions she plays after Player I plays $\xi$ as his rank move, are in $I^n_\xi$ for a suitable $n$. More exactly, suppose she has played a partial chain isomorphism $f$ of cardinality $\le\kappa_n$ in $I^n_\xi$ and then Player I plays a bounded set $A$ of cardinality $\le\kappa_m$, where $m>n$ as well as $\eta<\xi$ as the rank move. Player II uses (c) above to find $g\in I^m_\eta$ such that $f\subseteq g$ and $A\subseteq\dom(g)$. She plays this $g$ in the game.
$\eop_{\ref{d}}$
\end{proof}

\begin{lemma}\label{e}Let $\kappa$ be a singular cardinal of  cofinality $\omega$ and $\kappa=\sup_n\kappa_n$, where $\kappa_n<\kappa_{n+1}$.
The following conditions are equivalent:
\begin{enumerate}
\item Player II has a winning strategy in the game $EF^c_{\kappa,\infty}((M_n)_{n<\omega},(N_n)_{n<\omega})$.
\item There are sets $J_n$, $n<\omega$, such that
\begin{enumerate}
\item $J_n\ne\emptyset$.
\item if $f\in J_n$, then $f$ is a partial chain isomorphism  of cardinality $\le\kappa_n$  from $(M_n)$ to $(N_n)$,
\item if $f\in J_n$, then for all bounded $A\subseteq \bigcup_{n<\omega} M$ of cardinality $\le\kappa_m$, $m>n$, there is $g\in J_m$ such that $f\subseteq g$ and $A\subseteq\dom(g)$.
\item If $f\in J_n$, then for  all bounded $B\subseteq \bigcup_{n<\omega} N_n$  of cardinality $\le\kappa_m$, $m>n$,  there is $g\in J_m$ such that $f\subseteq g$ and $B\subseteq\ran(g)$.
\end{enumerate}
\end{enumerate}
\end{lemma}

\begin{proof}
 If (1) holds then we let $J_n$ be the set of partial chain isomorphisms  of cardinality $\le\kappa_n$  from $(M_n)$ to $(N_n)$ which are subsets of cardinality $\le\kappa_n$ of partial functions that arise in the game when II is playing her winning strategy. Conversely, suppose such sets $J_n$ exist. The strategy of II is to make sure that the partial functions she plays are in $J_n$ for a suitable $n$. More exactly, suppose she has played a partial chain isomorphism $f$ of cardinality $\le\kappa_n$ in $J_n$ and then Player I plays a bounded set $A$ of cardinality $\le\kappa_m$, where $m>n$. Player II uses (c) above to find $g\in J_m$ such that $f\subseteq g$ and $A\subseteq\dom(g)$. She plays this $g$ in the game.
$\eop_{\ref{e}}$
\end{proof}

\begin{proof} (of  Theorem~\ref{maxudwo}) We argue along the lines of the proof of 
Theorem 3.2 in \cite{Barwise}. Suppose that $\LL^\ast$ is as in the assumption and let $\varphi$ be an $\LL^\ast$-sentence. Suppose for a contradiction that 
for every $\alpha<\kappa^+$ there are chain models $(M^\alpha_n)_{n<\omega}$
and $(N^\alpha_n)_{n<\omega}$ in $\KK^c$ such that 
$(M^\alpha_n)_{n<\omega}\sim^c_{\kappa, \alpha}(N^\alpha_n)_{n<\omega}$, while 
$(M^\alpha_n)_{n<\omega}\models^c\varphi$ and 
$(N^\alpha_n)_{n<\omega}\models^c\neg \varphi$.
For each such $\alpha$ we get sets $I^n_\xi, \xi<\alpha,$ from Lemma~\ref{d}.

We now build a new $\LL^\ast$ sentence $\varphi^\ast$ in an extended vocabulary in which $U,P,Q$ are new unary relation symbols, $<^\ast$ is a new binary relation symbol and $I_n$ is a new $1+\kappa_n+\kappa_n$-ary relation symbol for each $n<\omega$. The sentence $\varphi^\ast$ is the conjunction of sentences saying:
\begin{enumerate}
\item the sentence $\varphi$
holds in the $P$ part of the (chain) model and $\neg\varphi$ holds in the $Q$-part,
\item   $\le^\ast$ is a linear order on $U$,

\item If $I_n(i,(a_\xi)_{\xi<\zeta},(b_\xi)_{\xi<\zeta})$, then $i\in U$, $(a_\xi)_{\xi<\zeta}\subseteq P$ and $(b_\xi)_{\xi<\zeta}\subseteq Q$, both bounded,
\item If $i<^*j$ and $I_n(j,\bar{a},\bar{b} )$, then  $I_n(i,\bar{a},\bar{b} )$ for all bounded sequences $\bar{a}$ and $\bar{b}$ of length $\kappa_n$,
\item For each $n<\omega$, if $I_n(i,(a_\xi)_{\xi<\kappa_n},(b_\xi)_{\xi<\kappa_n})$ then for all $j<^\ast i$, $m>n$, and all bounded $(a'_\xi)_{\kappa_n\le\xi<\kappa_m}\in P$ ($(b'_\xi)_{\kappa_n\le\xi<\kappa_m}\in Q$) there is a bounded $(b'_\xi)_{\kappa_n\le\xi<\kappa_m}\in Q$ ($(a'_\xi)_{\kappa_n\le\xi<\kappa_m}\in P$) such that 
$$I_m(j,(a_\xi)_{\xi<\kappa_n}\frown (a'_\xi)_{\kappa_n\le\xi<\kappa_m},(b_\xi)_{\kappa_n\le\xi<\kappa_m}\frown(b'_\xi)_{\kappa_n\le\xi<\kappa_m})).$$
\item If $I_n(i,(a_\xi)_{\xi<\kappa_n},(b_\xi)_{\xi<\kappa_n})$, then for all atomic formulas 
$\phi((v_\xi)_{\xi<\kappa_n})$ we have $$(M_m)_{m<\omega}\models^c\phi((a_\xi)_{\xi<\kappa_n})\iff (N_m)_{m<\omega}\models^c \phi((b_\xi)_{\xi<\kappa_n}),$$  where $(M_m)_{m<\omega}$ is the $P$-part of $(K_m)_{m<\omega}$ and $(N_m)_{m<\omega}$ is the $Q$-part.
\end{enumerate}

Using the fact that 
$L^c_{\kappa, \kappa}$ is included in $\LL^\ast$, we can express this by an
$\LL^\ast$-sentence, which we take to be $\varphi^\ast$. 

For every $\kappa\le\alpha<\kappa^+$ we build a chain model  $(K_n)^\alpha_{n<\omega}$ of $\varphi^\ast$ in which
$<^\ast$ has order type $\alpha$, by
declaring $K_0^\alpha$ to be the ordinal $\alpha$, interpreting $U$ as $K_0^\alpha$ and $<^\ast$ as the natural order on
$\alpha$, letting $K_{n+1}^\alpha=M_n^\alpha\bigcup
N^\alpha_n$ and declaring 
$\bigcup_{n<\omega} M^\alpha_n$ to be the $P$-part and $\bigcup_{n<\omega} N^\alpha_n$ to be the 
$Q$-part of the model. The interpretation of the predicates $I_n$ in $(K_m)^\alpha_{m<\omega}$
is as follows: $I_n(\gamma,(a_\xi)_{\xi<\zeta},(b_\xi)_{\xi<\zeta})$, $\zeta<\alpha$, is defined to hold in $(K_m)^\alpha_{m<\omega}$ if and only if 
the relation $\{(a_\xi,b_\xi):\xi<\zeta\}$ is in the set $I^n_\gamma$.

Now we are in the situation to apply $\udwo_\kappa$ of $\LL^\ast$.
Doing so, we obtain a strict chain model $(K_n)_{n<\omega}$ of $\varphi^\ast$ in which $<^\ast$ is not well-founded. Let $e_0>^*e_1>^*e_2\ldots$ be a descending chain in $<^*$ in $(K_n)_{n<\omega}$. Let $J_n$ consist of functions $f$ such that $f(a_\xi)=b_\xi$ for some pair $((a_\xi)_{\xi<\zeta},(b_\xi)_{\xi<\zeta})$ for which $(K_n)_{n<\omega}$  satisfies $I_n(e_m,(a_\xi)_{\xi<\zeta},(b_\xi)_{\xi<\zeta})$ for some $m<\omega$.
By Lemma~\ref{e}, Player II has a winning strategy in the game $EF^c_{\kappa,\infty}((M_n)_{n<\omega},(N_n)_{n<\omega})$, where $(M_n)_{n<\omega}$ is the $P$-part of $(K_n)_{n<\omega}$ and $(N_n)_{n<\omega}$ is the $Q$-part. By Lemma~\ref{Scottwo}, $(M_n)_{n<\omega} \cong^c  (N_n)_{n<\omega}$, contradicting the fact that $(M_n)_{n<\omega}  \models^c\varphi$ and $ (N_n)_{n<\omega} \models^c\neg\varphi$.
$\eop_{\ref{maxudwo}}$
\end{proof}

\begin{theorem}\label{realmax} Suppose that $\kappa$ is a strong limit of singular cardinals of countable cofinality, 
$\tau$ a fixed vocabulary and $\KK^c$ the class of all weak chain models of $\tau$. 
Suppose that $\LL^\ast$ is a logic on chain models such that:
\begin{itemize}
\item $L^c_{<\kappa, <\kappa}\le\LL^\ast$ and
\item $(\LL^\ast,\models^c, \KK^c)$ satisfies $\udwo_{<\kappa}$.
\end{itemize}
Then $\LL^\ast \le L^c_{<\kappa, <\kappa}$.
\end{theorem}

\begin{proof} Corollary \ref{lessudwo} shows that $L^c_{<\kappa, <\kappa}$ satisfies $\udwo_{<\kappa}$. Let
$\langle \lambda_i:\,i<i^\ast\rangle$ be an increasing sequence of singular cardinals of countable cofinality
converging to $\kappa$. Theorem \ref{maxudwo} shows that for every sentence $\varphi$ of $\LL^\ast$, there is
$i<i^\ast$ such that $\varphi$ is a definable in $L^{c}_{\lambda_i,\lambda_i}$, and therefore
$\varphi$ is a definable in $L^c_{<\kappa, <\kappa}$.
$\eop_{\ref{realmax}}$
\end{proof}

The above results in conjunction with Shelah's characterisation of $L^1_\kappa$ as maximal above 
$\bigcup_{\lambda<\kappa}L_{\lambda, \omega}$  to satisfy SUDWO, indicate that it is likely that SUDWO 
is a strict strengthening of UDWO. However, the conclusion cannot be drawn directly since Shelah's characterisation is obtained in the class of logics with the ordinary meaning of a model, which does not fit the chain logic. For all we know, SUDWO and UDWO might be equivalent.

The reader might now be curious to see the definition of SUDWO, so we recall its definition embedded in Problem 1.5. of \cite{Sh797} (see also Conclusion 3.2. of \cite{Sh797} ).

\begin{definition}\label{SUDWOdef} Let $\tau$ be a vocabulary of size $<\kappa$, including a constant symbol $c$ and
two place predicates $<$ and $R$.  Let $\psi$ be any $L^1_\kappa$ sentence in $\tau$.

Then $\psi$ does not define a well order, in the concrete sense that for any large enough $\theta<\kappa$ and large enough 
$\lambda$, for any given model $\mathfrak A$ of $\psi$ which is a $\tau$-expansion of the model $(\mathcal H(\lambda), \theta, \in, < )$
where $<^{\mathfrak A}$ extends the order of the ordinals in $\mathcal H(\lambda)$, $R^{\mathfrak A}$ extends the membership relation on $\mathcal H(\lambda)$, and $c^{\mathfrak A}=\theta$, there is a $\tau$-model $\mathfrak B$ such that:

\begin{itemize}
\item  $\mathfrak B$ is a model of $\psi$ which is a $\tau$-expansion of the model $(\mathcal H(\lambda), \theta, \in, < )$
and  $<^{\mathfrak B}$ extends the order of the ordinals in $\mathcal H(\lambda)$, $R^{\mathfrak B}$ extends the membership relation on $\mathcal H(\lambda)$, and $c^{\mathfrak B}=\theta$, 
\item 
there is a sequence $\langle a_n:\,n<\omega\rangle$ of subsets of the universe of ${\mathfrak B}$ (i.e. unary relations) such that
\begin{enumerate}
\item for every $n$, 
\[
 \mathfrak B\models (\forall x) [x R a_n\implies x R a_{n+1}],
 \]
\item for every $n$, 
\[
 \mathfrak B\models ``\theta\mbox{ is a cardinal such that }|a_n|\le\theta",
 \]
 (note that the above sentence is expressible in  $\mathfrak B$ since  $\mathfrak B$ is an expansion of $(\mathcal H(\lambda), \theta, \in, < )$) and 
 \item
 for every $b \in  \mathfrak B$ there is $n$ such that $\mathfrak B\models ``b\in a_n"$.
 \end{enumerate}
 \end{itemize}
\end{definition}

The properties of the model $\mathfrak B$ above clearly imply that $<^{\mathfrak B}$ is not a well order, as it violates the rule of cardinal arithmetic $\aleph_0\cdot \theta < \lambda$. Hence $\psi$ cannot characterise the models in which $<$ is a well order.

It is quite interesting that the defining property of SUDWO basically says that every rich enough model has a twin which looks like a strange chain model. We have not been able to make more of this remark.

\section{Further properties of chain logic and Shelah's Problem 1.4.}\label{sec:further} Among other results, in this section we shall show that chain logic satisfies 
the requirements of Shelah's Problem 1.4.

We first discuss (in)compacteness properties of the chain logic, in \S\ref{incomp}. This leads us to a method to compare the chain logic 
to model classes of $L^{c,\ast}_{\kappa,\kappa}$, Theorem \ref{PCtheorem}. Using this and Theorem 2.6 of Cunnigham \cite{Cunnigham} who showed that 
$L^{c,\ast}_{\kappa,\kappa}$ has  Interpolation (for a definition of Interpolation, see Definition \ref{interpolation}) we show 
in \S{shelahproblem}
that chain logic answers Shelah's Problem, modulo a model class, see Corollary \ref{solves1.4.}.

In addition to the above results, we also note (\S\ref{theunion}) that the chain logic satisfies the Union Lemma. Finally and somewhat surprisingly, in \S\ref{chainind}, we show that the chain independent fragment of the chain logic is quite non-trivial. 

\subsection{The incompactness of $L_{\kappa,\kappa}^c$}\label{incomp} A property that does not seem to have been studied 
classically for  chain logic is its degree of compactness. Positive compactness results have been  obtained for restricted classes of chain models: Green \cite{green1974} proved a $\Sigma_1$--compactness result for admissible sets, which was given another proof by Makkai \cite{Makkaichainlogic}. Possibly Karp and her school were aware of the incompactness of  full chain logic, but in the absence of any literature on it, we give a proof. In fact we give two proofs, as the proofs for the logics $L^c_{\kappa, \kappa}$ and  $L^{c, \ast}_{\kappa, \kappa}$ turn out to be different.

It is not straightforward to determine all cardinals $\lambda$ such that the logic $L_{\theta, \omega}$ is 
$\lambda$-compact or satisfies some weak version of $\lambda$-compactness. Clearly, if $\theta=\lambda$ is a strongly compact cardinal, this is the case. However, Boos \cite{Boss} showed that it is even possible to have analogues of weak compactness of $L_{\theta, \omega}$ for $\theta$ below the continuum (though $\theta$ needs to be a large cardinal in $L$).\footnote{Moreover, Stavi \cite{StaviI} showed that below a measurable cardinal this is true for basically any logic.} 
By the results of Keisler and Tarski \cite{KeislerTarski}, see Theorem 3.1.4 of \cite{Dickmann}, $L_{\kappa, \omega}$ is never $\kappa$-compact for $\kappa$-singular. Moreover, see Corollary 3.1.6 in \cite{Dickmann}, $L_{\lambda, \omega}$ is $\lambda$-compact iff $\lambda$ is a strongly compact cardinal. 

\begin{corollary}\label{notcompact} The logic $L^{c,\ast}_{\kappa,\kappa}$ is not $\kappa$-compact. \end{corollary}

\begin{proof} Theorem \ref{lowerbound} shows that $L^{c,\ast}_{\kappa,\kappa}$ is Chu above $L_{\kappa, \omega}$ and so by Corollary \ref{kappacompactnesspreservation}, $\kappa$-compactness of $L^{c,\ast}_{\kappa,\kappa}$ would imply the $\kappa$-compactness of $L_{\kappa, \omega}$, a contradiction.
$\eop_{\ref{notcompact}}$
\end{proof}

Theorem \ref{lowerbound} does not immediately apply to the logic $L^{c}_{\kappa,\kappa}$ so it does
not rule out some degree of compactness of this logic. If $L_{\kappa,\omega}$ had some degree of upwards L\"owenheim-Skolem theorem, we would be able to replace ${\mathcal M}_{\ge \kappa}$ in Theorem~\ref{lowerbound} (2) by ${\mathcal M}$ and draw the same conclusion for $L^{c}_{\kappa,\kappa}$ as we did for 
$L^{c,*}_{\kappa,\kappa}$. However, there is no sufficiently strong such upwards theorem.\footnote{See Chapter IV of \cite{Dickmann}.} For example, the following is a well known open question:

\begin{question}\label{} If a complete $L_{\omega_1,\omega}$-sentence has a model of size $\aleph_{\omega_n}$ for every $n$, does it then have a model of size $\aleph_{\omega_\omega}$?
\end{question}

Therefore, we need another argument to deal with the question whether the chain logic $L_{\kappa,\kappa}^c$ is
$\kappa$-compact.

\begin{theorem}\label{incompactness} $L_{\kappa,\kappa}^c$ is not $\kappa$-compact.
\end{theorem}

\begin{proof} We  construct a set $\Gamma$ of $L_{\kappa,\kappa}^c$-sentences which is
$(<\kappa)$-satisfiable but not $\kappa$-satisfiable. We need the following lemma.

\begin{lemma}\label{ordertypekappa} There is an $L_{\kappa,\omega}$-sentence $\theta$ in the language consisting of one binary predicate $<^\ast$ whose models are exactly the models of size $\kappa$, where $<^\ast$ is a well order of
the domain of order type $\kappa$.
\end{lemma}

\begin{Proof of the Lemma} By Observation \ref{moreforless}, it suffices to find an $L_{\kappa^+,\omega}$-sentence $\theta$ with the required properties. 
Following Dana Scott \cite{Scottinfinitary}, starting from the first order sentence stating that $<^\ast$ is a linear order, one can define by induction on 
$\alpha<\kappa$ an $L_{\kappa^+,\omega}$ formula $E_\alpha(x)$ so that
\[
E_\alpha(x)\iff(\forall y)[ y <^\ast x \iff \bigvee_{\beta <\alpha} E_\beta (y)].
\]
Then one shows by induction on $\alpha$ that $E_\alpha(x)$ means that the $<^\ast$-predecessors of $x$ form
a well-order of order type $\alpha$. The sentence for our proof is 
$$\bigwedge_{\alpha<\kappa}(\exists x)E_\alpha(x)\wedge(\forall x) \bigvee_{\alpha<\kappa} E_\alpha (x).$$
$\eop_{\ref{ordertypekappa}}$
\end{Proof of the Lemma}

Now we formulate $\Gamma$ in the language $\{<^\ast\}\cup \{c_\alpha:\,\alpha<\kappa\}\cup \{d\}$ as
\[
\Gamma=\{\theta\}\cup \{c_\alpha <^\ast c_\beta:\,\alpha<\beta<\kappa\}\cup \{c_\alpha <^\ast d:\,\alpha<\kappa\}.
\]
Then for any subset $\Gamma_0$ of $\Gamma$ of size $<\kappa$, any model $M$ of $\theta$ easily gives rise to an
actual model of  $\Gamma_0$ by interpreting the relevant $c_\alpha$ as the $\alpha$-th element of $M$
in the well order provided by $<^\ast$ and $d$ as any element of $M$ which is of large order in $M$ than any of these relevant $c_\alpha$s. This model is also a chain model, since all sentences in $\Gamma_0$ are
in $L_{\kappa^+,\omega}$. On the other hand, $\Gamma$ does not have any models or chain models.
$\eop_{\ref{incompactness}}$
\end{proof}

We now turn to upper bounds for the compactness number of the chain logic. 

\begin{convention}\label{PC} 
Let $\tau$ be a vocabulary.

{\noindent (1)} We define a new vocabulary $\tau'=\tau\cup \{P_n:\,n<\omega\}$ where
$P_n$ are new unary predicate symbols. 

{\noindent (2)} For each $n<\omega$ we define an $L_{\kappa,\kappa}$-sentence $\psi_n$ in the vocabulary $\tau'$ as follows
\[
(\exists (x_\alpha)_{\alpha<\kappa_n}) (\forall y) [P_n(y)\implies \bigvee_{\alpha<\kappa_n} y=x_\alpha].
\]
(This sentence expresses that the realization of the predicate $P_n$ has size $\le\kappa_n$.)

{\noindent (3)} 
Let $\sigma_0$ be the $L_{\kappa,\kappa}$-sentence 
\[
(\forall x ) \bigvee_{n<\omega}P_n( x) \wedge  \bigwedge_{n<\omega}(\forall  x )[P_n( x)\implies P_{n+1}( x)].
\]
(Notice that if this sentence is true for any $x$ then it is trivially true for finitary ${\bar x}$, but this is not necessarily the case for ${\bar x}$ of infinite length).
Let $\sigma_1$ be $\sigma_0\wedge \bigwedge_{n<\omega}\psi_n$.

{\noindent (4)} 
For $l<2$, we define ${\mathcal Mod}(\sigma_l)$ to be the class of all $L_{\kappa,\kappa}$-models of $\sigma_l$.
\end{convention}

\begin{theorem}\label{PCtheorem} 
To avoid problems with extending vocabularies, for this theorem we assume that all statements are relativized to 
a vocabulary $\tau'$ as in convention \ref{PC}, in particular, $\tau'$ was obtained from a vocabulary $\tau$
as in that convention. Then:

{\noindent (1)} (a) $L^{c,\ast}_{\kappa,\kappa} \le (L_{\kappa,\kappa}, \models, {\mathcal Mod}(\sigma_0))$ and
(b) $L^{c}_{\kappa,\kappa} \le (L_{\kappa,\kappa}, \models, {\mathcal Mod}(\sigma_1))$.

\smallskip

{\noindent (2)} For models $\ma$ and $\mb$ in ${\mathcal Mod}(\sigma_0)$ let $\fallingdotseq$ be the equivalence relation saying that the reducts of $\ma$ and $\mb$ to $\tau$ are the same.  Then, in addition to
(1)(a), 
\[
L^{c,\ast}_{\kappa,\kappa} \le_i^\fallingdotseq (L_{\kappa,\kappa}, \models, {\mathcal Mod}(\sigma_0)).
\]
\end{theorem}

\begin{proof}  {\noindent (1)} (a) For the first inequality, 
we define $f(\psi)$ for sentences $\psi$ in $L^{c,\ast}_{\kappa,\kappa}$ by induction on the complexity of $\psi$.
For atomic sentences we have $f(\psi)=\psi$ and this is extended through the logical operations $\neg$ and
conjunctions and disjunctions of size $<\kappa$. It remains to define  $f(\psi)$ for $\psi$ of the form
$\forall {\bar x}\, \varphi (  {\bar x};  {\bar y})$ where ${\bar y}$ is a sequence of parameters and $\bar{x}=
\langle x_\alpha:\alpha<\alpha^\ast\rangle $ for some $\alpha^\ast<\kappa$.

We let $f(\forall  \bar{x} \,\varphi (   \bar{x};  {\bar y}))$ be 
the sentence $\forall  \bar{x} ([\bigvee_{n<\omega} \bigwedge_{\alpha<\alpha^\ast}P_n (x_\alpha)]\implies \varphi (   x;  {\bar y}))$.

To define $g$, let $M$ be a model of $\sigma_0$
and define $g(M)=(P_0^M, P_1^M, \ldots)_{n<\omega}$, which gives a weak chain model by the choice of $\sigma_0$. The pair $(f,g)$ is well defined, $f$ preserves the logical operations of $L_{\kappa,\kappa}$ and $g$ is onto. It remains to verify the adjointness condition, which will imply that $(f,g)$ is indeed a Chu transform.

Suppose that $\psi$ is in $L^{c,\ast}_{\kappa,\kappa}$ and $M\in S_0'$ is such that $M\models f(\psi)$.
By induction on the complexity of $\psi$ we have to check that $(P_0^M, P_1^M, \ldots)_{n<\omega}\models^c \psi$. The nontrivial case of this induction is the quantifier case, so suppose that 
$\psi$ is of the form $\forall {\bar x}\, \varphi (  {\bar x};  {\bar y})$, where ${\bar y}$ bounded, that is, ${\bar y}$ is such that there is $n$ for which all individual elements of $y$ are in $P_n^M$.
Let $\bar{x}$ also be bounded.
The definition is made so that  $M\models f(\psi)$ implies that $M\models \varphi ( {\bar x};  {\bar y})$, and therefore, since ${\bar y}$ is bounded we have that $g(M)=(P_0^M, P_1^M, \ldots)_{n<\omega}\models^c \forall {\bar x}\, \varphi (  {\bar x},  {\bar y})$.

{\noindent (b)} The definition of $f$ remains the same for this proof. The definition of $\sigma_1$ is such as to
guarantee that for every $M\in S_1'$ and for every $n<\omega$ we have $|P_n|^M\le \kappa_n$. Therefore, leaving the definition of $g$ the same as in the proof of (a), we have that $g(M)=(P_0^M, P_1^M, \ldots)_{n<\omega}$
is a strict chain model. The proof of the adjointness property remains the same. However, $g$ is not necessarily onto, as it is perfectly possible to have chain models
$(A_n)_{n<\omega}$ for which there are $n$ that satisfy $|A_n|>\kappa_n$. We shall however show that $g$ satisfies the density condition.

So suppose that $\varphi$ is a sentence of $L^{c}_{\kappa,\kappa}$ which has a proper chain model 
$(A_n)_{n<\omega}$. We shall define $B_n$ by induction on $n<\omega$. For convenience let us define
$A_1=\emptyset$. Let $B_0$ be $A_k$ for the largest $k$ such that $|A_k|\le \kappa_0$. Such a $k$ exists because the sets $A_k$ are increasing and $|\bigcup_{n<\omega} A_n|=\kappa$. Having defined $B_n$ as 
$A_{k_n}$ with $|A_{k_n}|\le \kappa_n$ and $k_n$ the largest $k$ with that property, let $B_{n+1}=A_k$
for the largest $k$ such that $|A_k|\le \kappa_{n+1}$. In this way we necessarily have $B_n\subseteq B_{n+1}$.
The chain model $(B_n)_{n<\omega}=g(M)$ for $M$ in $S'_1$ which is obtained on 
$\bigcup_{n<\omega} (B_n)$ by interpreting each $P_n$ as the corresponding $B_n$. By the same argument as in Observation \ref{variouschainmodels}(3), $(B_n)_{n<\omega}$ is a model of $\varphi$.

{\noindent (2)} We use the same $(f,g)$ as in (1)(a). Suppose that we are given a chain model 
$(A_n)_{n<\omega}$ and $\ma\in {\mathcal Mod}(\sigma_0)$ with $\ma=\bigcup_{n<\omega} A_n$.
Define $\mb$ to be the the model whose reduct to $\tau$ is the same as that one of $\ma$, but 
$P^{\mb}_n=A_n$ for each $n$. Then it is easily seen that $\mb\fallingdotseq \ma$ and $g(\mb)= 
(A_n)_{n<\omega}$.
$\eop_{\ref{PCtheorem}}$
\end{proof}

The following is a classical argument resembling the projective class arguments, for example 
Proposition II 3.1.9 in  \cite{BarwiseFeferman}.

\begin{theorem}\label{Oiko} Suppose that $\sigma$ is in $L_{\kappa,\kappa}$ and that $\theta$ is a cardinal such that $L_{\kappa,\kappa}$ is $\theta$-compact. Then so is $(L_{\kappa,\kappa}, \models, {\mathcal Mod}(\sigma))$.
\end{theorem}

\begin{proof} Suppose that $\Gamma$ is a set of sentences of $L_{\kappa,\kappa}$ such that
every subset of $\Gamma$ of size $<\theta$ is satisfiable in ${\mathcal Mod}(\sigma)$. In particular, 
$\Gamma\cup \{\sigma\}$ is a 
$(<\theta)$-satisfiable set of sentences of $L_{\kappa,\kappa}$ in the vocabulary $\tau'$ obtained by enlarging $\tau$ by the symbols in $\sigma$. Hence, by the assumption, $\Gamma\cup \{\sigma\}$ has a model $M$. Then $M\in {\mathcal Mod}(\sigma)$ by definition.
$\eop_{\ref{Oiko}}$
\end{proof}

\begin{corollary}\label{strcompact} $L^{c,\ast}_{\kappa,\kappa}$ and $L^{c}_{\kappa,\kappa}$  are $\theta$-compact whenever $\theta$ is
a strongly compact cardinal $>\kappa$.
\end{corollary}

\begin{proof} By the definition of a strongly compact cardinal, $L_{\theta, \theta}$ is $\theta$-compact. Let 
$\Gamma$ be a $(<\theta)$-satisfiable set of 
 $L_{\kappa,\kappa}$-sentences. Then $\Gamma$  is in particular a set of $L_{\theta, \theta}$-sentences and hence has a model in the sense of $L_{\theta, \theta}$.
By the definition of $\models$ in the two logics, this model is also an $L_{\kappa,\kappa}$ model. So, 
$L_{\kappa,\kappa}$ is $\theta$-compact. (In fact, all this just says that $L_{\kappa, \kappa}\le L_{\theta, \theta}$). By Theorem \ref{Oiko},  $(L_{\kappa,\kappa},\models, {\mathcal Mod}(\sigma_l))$ for $l<2$ are 
$\theta$-compact. Then, by Theorem \ref{PCtheorem} and Corollary \ref{kappacompactnesspreservation}, 
$L^{c,*}_{\kappa,\kappa}$ and $L^{c}_{\kappa,\kappa}$ are $\theta$-compact.
$\eop_{\ref{strcompact}}$
\end{proof}

Finally, we show that, as in the case of
$L^2_{\omega,\omega}$, the second order logics $(L^{c}_{\omega,\omega})^2$ and $(L^{c, \ast}_{\omega,\omega})^2$ are not compact, or even $(\omega,\omega)$-compact. The definition of these logics is the same as in  ordinary chain logic, except that we now consider all second order sentences of $L^2_{\omega, \omega}$, only allowing ordinary or set variables which are bounded by an element of the chain.

\begin{theorem}\label{non-compactness} Neither $(L^{c}_{\omega,\omega})^2$ nor $(L^{c, \ast}_{\omega,\omega})^2$ are $(\omega,\omega)$-compact.
\end{theorem}

\begin{proof}
Let $\phi$ be the second 
order sentence 
$$\exists X\forall y\, (P(y)\to X(y)).$$ 

\begin{lemma}\label{claim1}The chain models of $\phi$ are the chain models in which $P$ is contained in one level of the chain.
\end{lemma}

\begin{proof} Suppose that $M=(M_n)_{n<\omega}$ as above is a chain model of $\phi$.
 Then there is an interpretation $Q$ of the variable $X$ such that $P^M\subseteq Q$. By the truth definition of chain models the variable $X$ ranges over subsets of $M$ which are  included in some $M_n$. Thus $Q\subseteq M_n$ for some $n$. Hence $P^M\subseteq M_n$.  For the other direction, suppose that $M=(M_n)_{n<\omega}$ is a chain model in which $P^M$ is contained in some $M_n$. Then we can pick $P^M$ as the interpretation of $X$ and $\phi$ becomes true.
\end{proof}

\begin{lemma}\label{claim2}
In models of $\phi$ we have full second order quantification over subsets of 
$P$. 
\end{lemma}

\begin{proof}
Suppose $M=(M_n)_{n<\omega}$ is a chain model of $\phi$. By Lemma \ref{claim1}, 
there is $n$ such that $P^M\subseteq M_n$. If we relativize second order quantification to $P$ as follows:
$$\exists X\psi(X)\mapsto\exists X[(\forall y(X(y)\to P(y))\wedge\psi(X)]$$
$$\forall X\psi(X)\mapsto\forall X[(\forall y(X(y)\to P(y))\to\psi(X)]$$
then we can  allow the standard interpretation for second order quantifiers because the relativization forces the bound second order variables to range over subsets of $P$, and hence to be always subsets of $M_n$.
\end{proof}

To finish the proof, let $\theta$ be the second order sentence which says that $<$ is a well-order on the predicate $P$. In chain models of $\theta$ we have no guarantee that $<$ is really a well-order because a descending sequence may cross over all the sets $M_n$. In models of 
$\theta\wedge \phi$ we know by the above that $<$ is really a well-order because any potential descending chain is a subset of $P$ and hence a subset of some $M_n$. We can now form a finitely consistent theory  $\{\theta,\phi\}\cup\{c_0>c_1>c_2>\ldots\}\cup \{P(c_n):\, n<\omega\}$, which has no models.
$\eop_{\ref{non-compactness}}$
\end{proof}

\subsection{Shelah's Problem 1.4}\label{shelahproblem}
Let us recall the definition of Craig's Interpolation:

\begin{definition}\label{interpolation} (1) For a logic $\LL$ and sentences $\rho_0, \rho_1$ we say $\rho_0\models_\LL \rho_1$ if for any model $M$ of $\LL$ in a vocabulary including those of $\rho_0$ and $\rho_1$, we have \[M\models_\LL \rho_0\implies M\models_\LL \rho_1.\]
{\noindent (2)} A logic $\LL$ is said to satisfy {\em Interpolation} if for any two vocabularies $\tau_0$ and $\tau_1$ and sentences $\rho_l\in \LL(\tau_l) (l <2)$, if $\rho_0\models_\LL \rho_1$ then there is a sentence $\rho$ in $\LL(\tau_0\cap \tau_1)$ such that $\rho_0\models_\LL \rho $ and $\rho\models_\LL \rho_1$.
\end{definition}

Shelah \cite{Sh797} showed that $L^1_\kappa$ satisfies Interpolation. This motivated him to ask:

{\bf Problem 1.4. from \cite{Sh797}}: Suppose that $\kappa$ is a singular strong limit of countable cofinality. Is there a logic between $ L_{\kappa^+, \omega}$ and 
$L_{\kappa^+, \kappa^+}$ which satisfies Interpolation ?

The logic $L^1_\kappa$ does not solve Problem 1.4. from \cite{Sh797}, since it is not above $L_{\kappa^+,\omega}$. We show that, modulo a model class, chain logic does. That actually follows easily from the above results and the earlier work of Cunnigham \cite{Cunnigham} who showed various versions of interpolation for chain logic, including Craig, Malitz and Lyndon interpolation. Her proofs are quite different from Shelah's proof for $L^1_\kappa$, because she uses the method of 
chain consistency properties, which has not yet been applied to $L^1_\kappa$ but is well developed for chain logic through the work of Karp and Cunnigham. This method is quite powerful and very close to showing that a logic is complete since it consists of recursively isolating sets of sentences that capture all sentences that are true in a certain class of models.

\begin{corollary}\label{solves1.4.} Suppose that $\kappa$ is a singular strong limit cardinal of countable cofinality. Then the chain logic $L^{c,\ast}_{\kappa,\kappa}$ satisfies 
\[
L_{\kappa^+,\omega}\le L^{c,\ast}_{\kappa,\kappa} \le (L_{\kappa^+,\kappa^+}, \models, {\mathcal Mod}(\sigma_0))
\]
and has the Interpolation, so modulo a model class, it gives a solution to Problem 1.4. from \cite{Sh797}.
\end{corollary}

\begin{proof} Theorem 2.6 of Cunnigham \cite{Cunnigham} shows that 
$L^{c,\ast}_{\kappa,\kappa}$ has  Interpolation. The Chu inequalities are proved in Theorem \ref{lowerbound}  and Theorem \ref{PCtheorem} respectively (note also Observation \ref{moreforless}).
$\eop_{\ref{solves1.4.}}$
\end{proof}

\subsection{The Union Lemma}\label{theunion}
The Union Lemma  features in several proofs about $L^1_\kappa$, see for example Remark 1.17. (b) 
in \cite{Sh797}\footnote{ It says that some concessions were made in the definition of $L^1_\kappa$ in order to have the Union Lemma hold for $L^1_\kappa$.} We show that the Union Lemma holds for  chain logic.

\begin{definition}\label{}
For chain models $(A_n)_{n<\omega}$ and $(B_n)_{n<\omega}$ we define 
\begin{enumerate}
\item $(A_n)_{n<\omega} \subseteq (B_n)_{n<\omega}$
if $\bigcup_n A_n\subseteq\bigcup_n  B_n$ and $\forall n\exists m\,(A_n\subseteq B_m)$
\item $(A_n)_{n<\omega} \prec_{L^c_{\kappa, \kappa}} (B_n)_{n<\omega}$ if 
 $(A_n)_{n<\omega} \subseteq (B_n)_{n<\omega}$
and for any bounded sequence $\bar{a}$ in $(A_n)_{n<\omega}$ and formula $\phi(\bar{x})$ in  $L^c_{\kappa, \kappa}$ we have
$$(A_n)_{n<\omega} \models^c \phi(\bar{a})\Leftrightarrow (B_n)_{n<\omega} \models^c\phi(\bar{a}).$$
\end{enumerate}

\end{definition}

\begin{observation}\label{} Suppose $(A_n)_{n<\omega} \subseteq (B_n)_{n<\omega}$. Then: \begin{enumerate}
\item If $\bar{a}$ is bounded in $(A_n)_{n<\omega}$, then it is bounded in $(B_n)_{n<\omega}$.
\item If $(A_n)_{n<\omega}\models^c\phi$ and $\phi$ is existential sentence, then $(B_n)_{n<\omega}\models^c\phi$. 
\end{enumerate}
\end{observation}

\begin{definition}\label{union} Suppose that for each $m<\omega $ we have a chain model $(A^m_n)_{n<\omega}$ 
and that these models satisfy  $(A^m_n)\subseteq (A^{m+1}_n)$ for all $m$. 
Let $A^\omega_n=\bigcup_{m\le n} A^m_n$, for each $n$.

We define the \emph{union} of the sequence $\langle (A_n)_{n<\omega}:\,m<\omega\rangle $, as the chain model 
$(A^\omega_n)_{n<\omega}$.
\end{definition}

\begin{observation}\label{} Using the notation of Definition \ref{union}, we have that  
\[
(A^m_n)_{n<\omega}\subseteq (A^{\omega}_n)_{n<\omega}
\]
for all $m$.
\end{observation}

Putting these observations together, we easily obtain that the chain logic satisfies the following Union Lemma.
The Union Lemma for $L^1_\kappa$, in contrast, is a difficult theorem.

\begin{theorem}[Union Lemma]\label{unionlemma}  If $(A^m_n)_{n<\omega}\prec_{L^c_{\kappa, \kappa}}(A^{m+1}_n)_{n<\omega}$ for all $m<\omega$, then 
$(A^m_n)_{n<\omega}\prec_{L^c_{\kappa, \kappa}}(A^{\omega}_n)_{n<\omega}$.
\end{theorem}

\begin{proof} By the usual Tarski-Vaught argument, noticing that every bounded sequence in the union
is bounded in one of the models.
$\eop_{\ref{unionlemma}}$
\end{proof}

\subsection{The chain-independent fragment}\label{chainind}
A point in favour of chain models seems to be the difference between the views that various chain decompositions of the same model can take on the same sentence. The following theorem renders that intuition formal.
To avoid trivialities, we formulate it in terms of proper chain models.

\begin{definition}\label{chainindependent} An $L_{\kappa, \kappa}$-sentence $\varphi$ is said to be
{\em chain-independent} if
 for any proper chain models 
$(A_n)_{n<\omega}, (B_n)_{n<\omega}$ we have that
\[
\bigcup_{n<\omega}A_n=\bigcup_{n<\omega}B_n \implies  [(A_n)_{n<\omega}\models^c \varphi\iff 
 (B_n)_{n<\omega}\models^c \varphi].
 \]
 The {\em chain-independent fragment of} $L_{\kappa, \kappa}$ is the set
of chain-independent sentences of $L_{\kappa, \kappa}$.
\end{definition}

It is obvious that $L_{\kappa\omega}$ is included in the chain-independent fragment of 
$L_{\kappa, \kappa}$.

\begin{theorem}\label{nondependent}   An $L_{\kappa, \kappa}$-sentence $\varphi$ is chain-independent iff for any model $\ma$ of size $\kappa$
 \[
 \ma\models \varphi\iff \mbox{ for any proper chain model }(A_n)_{n<\omega}\, [\bigcup_{n<\omega}\, A_n=\ma\implies  (A_n)_{n<\omega}\models^c \varphi].
\]
\end{theorem}

\begin{proof} For the proof we require the following lemma.

\begin{lemma}\label{helpful} Suppose that $\ma$ is a $\tau$ model and $\varphi$ an $L_{\kappa, \kappa}$-sentence such that $(A_n)_{n<\omega}\models^c \varphi$ for any chain decomposition 
$\langle A_n:\, n<\omega\rangle$ of $\ma$. Then $\ma\models \varphi$.
\end{lemma}

\begin{proof-} {\bf of Lemma \ref{helpful}.} By induction on the complexity of $\varphi$.
$\eop_{\ref{helpful}}$
\end{proof-}

For  the proof of the Theorem, first suppose that $\varphi$ is chain-independent, $\ma\models \varphi$ and $(A_n)_{n<\omega}$ 
is a proper chain decomposition of  $\ma$. If $(A_n)_{n<\omega}\nmodels ^c \varphi$ then by the definition of
chain-independence it follows that for any proper chain decomposition $(B_n)_{n<\omega}$ of  $\ma$ we have 
$(B_n)_{n<\omega}\models ^c \neg\varphi$. By Lemma \ref{helpful} it follows that $\ma\models \neg\varphi$,
a contradiction.

In the other direction, suppose that $\varphi$ satisfies the 
equivalence, but $\varphi$ is not chain-independent.
Then there are proper chain decompositions $(A_n)_{n<\omega}$ and $(B_n)_{n<\omega}$ of $\ma$
such that $(A_n)_{n<\omega}\models^c \varphi$ and $(B_n)_{n<\omega}\models^c \neg\varphi$, a contradiction.
$\eop_{\ref{nondependent}}$
\end{proof}

Theorem \ref{nondependent} can be used to show that the chain-independent fragment can express  a number
of classical concepts:

\begin{theorem}\label{Aronszajn} Suppose that $\lambda<\kappa$ and $\cf(\lambda)>\aleph_0$. The following properties can be expressed by chain-independent sentences of $L_{\kappa, \kappa}$.

\begin{itemize}
\item $G$ is a graph which omits a clique of size $\lambda$,
\item $T$ is a tree which omits a branch of length $\lambda$,
\item $<$ is a linear order with no decreasing sequences of length $\lambda$.
\end{itemize}

\end{theorem}
\begin{proof}
The proof is basically the same for all three examples, so let us show it for the case of graphs. Let $\rho$ be an $L_{\kappa, \kappa}$ sentence which says that $R$ is a binary relation and that it has no clique of 
size $\lambda$. We claim that $\rho$ is chain-independent.

Let $\ma$ be a model of $\rho$ size $\kappa$ and let $(A_n)_{n<\omega}$ a proper chain model with
$\bigcup_{n<\omega} A_n=\ma$. If $\ma\models\rho$ then clearly no $A_n$ can have a clique of size $\lambda$.
Suppose on the other hand that no $A_n$ has a clique of size $\lambda$, but that $\ma$ has such a clique, say $K$. Since $\cf(\lambda)>\aleph_0$ and $\ma$, there will be $n$ such that $K\cap A_n$ has size 
$\lambda$. Therefore $(A_n)_{n<\omega}\models^c \neg \rho.$
$\eop_{\ref{Aronszajn}}$
\end{proof}

\section{Concluding remarks}\label{fin}
We have shown that in the case of a singular cardinal $\kappa$ of countable cofinality which in addition satisfies $\kappa=\beth_\kappa$, the chain logic $L^{c,\ast}_{\kappa, \kappa}$ satisfies all the motivating requirements 
for the introduction of the logic $L^1_\kappa$, including solving Problem 1.4 from \cite{Sh797} which was left open by $L^1_\kappa$. 
A caveat to note, however, is that $L^1_\kappa$ and $L^{c,\ast}_{\kappa, \kappa}$ have a major difference in what they consider a structure: the former has the classical concept of a structure while the latter is based on chain models. On the other hand, $L^{c,\ast}_{\kappa, \kappa}$ has a classical syntax, while this is not the case with $L^1_\kappa$.

One could look for further desirable properties of a logic at a singular cardinal, most importantly the property of 
$\kappa$-compactness. Neither $L^1_\kappa$ nor  chain logic have it.

We finish by some philosophical remarks about the nature of the work done in the paper. 
We solved Shelah's problem, admittedly changing it along the way, by moving from classical logics to chain logics. However, our position has been to study Karp's old and very nice chain logic, which we have already studied in the context of singular cardinals. Here, we succeeded to further emphasise its canonical nature by characterising it among all chain logics. The characterisation is very much in the spirit of Lindstr\"om's characterisation of first order logic. This complements a programme which was already started by Barwise 
\cite{bairwiseablogic} who extended Lindstr\"om's proof to infinitary logics by means of a rather strong assumption called Karp Property. Shelah found a characterisation for his infinitary logic $L^1_\kappa$ in the context of singular cardinals and was able to avoid the use of Karp Property. In the process Shelah gave up the requirement, which some might find natural, that $L^1_\kappa$ should have an easily recognisable syntax. We allow the change to chain models and can achieve a Lindstr\"om-like proof without without giving up the syntax. In the context of studying Karp's chain logic, which has already been a subject of study, our results seem to emphasise its overall naturality.

\bibliographystyle{plain}
\bibliography{JMbiblio}

\end{document}